\documentclass[reqno]{tpms-lmyversion}

\usepackage{url}
\usepackage{eurosym}
\usepackage{enumitem}
\usepackage{graphicx}
\usepackage{epstopdf}
\usepackage[justification=centering]{caption}
\usepackage{subfigure}
\usepackage{bbm}
\usepackage{dsfont}

\newtheorem{thm}{Theorem}[section]

\newtheorem{prop}[thm]{Proposition}
\newtheorem{cor}[thm]{Corollary}

\theoremstyle{definition}
\newtheorem{defn}[thm]{Definition}
\newtheorem{example}[thm]{Example}

\newcounter{MyStepNo}
\DeclareRobustCommand{\mystep}[1]{%
   \refstepcounter{MyStepNo}%
   \theMyStepNo\label{#1}}

\newcounter{algoStepNo}

\theoremstyle{remark}

\numberwithin{equation}{section}

\begin{document}

\title[High-Roller Impact]{High-Roller Impact: A Large Generalized Game Model of Parimutuel Wagering}


\author{Erhan Bayraktar}
\address{Department of Mathematics, University of Michigan, Ann Arbor, Michigan 48109}
\email{erhan@umich.edu}
\thanks{This work is supported by the National Science Foundation under DMS-1613170. We gratefully acknowledge the anonymous referees for their valuable advice on improving the paper.}

\author{Alexander Munk}
\address{Department of Mathematics, University of Michigan, Ann Arbor, Michigan 48109}
\email{amunk@umich.edu}





%

\begin{abstract}
How do large-scale participants in parimutuel wagering events affect the house and ordinary bettors? A standard narrative suggests that they may temporarily benefit the former at the expense of the latter. To approach this problem, we begin by developing a model based on the theory of large generalized games. Constrained only by their budgets, a continuum of diffuse (ordinary) players and a single atomic (large-scale) player simultaneously wager to maximize their expected profits according to their individual beliefs. Our main theoretical result gives necessary and sufficient conditions for the existence and uniqueness of a pure-strategy Nash equilibrium. Using this framework, we analyze our question in concrete scenarios. First, we study a situation in which both predicted effects are observed. Neither is always observed in our remaining examples, suggesting the need for a more nuanced view of large-scale participants.
\end{abstract}

\maketitle

\section{Introduction}\label{intro sect}

Suppose that a collection of bettors are wagering on an upcoming event. The payoffs are determined via a {\it (frictionless) parimutuel system} if whenever Outcome $i$ occurs, Bettor $A$ receives 
\begin{align*}
\left(\text{Total Amount Wagered} \right) \left( \frac{ \text{Bettor $A$'s Wager on Outcome $i$}}{\text{Total Amount Wagered on Outcome $i$}} \right)  . 
\end{align*}
The idea is that players with correct predictions will proportionally share the final betting pool. Prizes are reduced in practice by transaction costs such as the {\it house take}, a percentage fee collected by the {\it betting organizer} (or {\it house}). For example, Bettor $A$ might only win
\begin{align}\label{friction pmg payoff words}
\kappa \left(\text{Total Amount Wagered} \right) \left( \frac{ \text{Bettor $A$'s Wager on Outcome $i$}}{\text{Total Amount Wagered on Outcome $i$}} \right)  
\end{align}
when Outcome $i$ occurs, if the house take is $\left( 1- \kappa \right) \%$ for some $0 < \kappa <1$.

This mechanism was invented in the context of horse race gambling by Oller in the late 1800's (\cite{oller+bio}) and remains widely employed in that setting: In 2014, worldwide parimutuel betting on horse races totaled around seventy-five billion euros (\cite{horse+racing+fed}). It also typically determines wagering payoffs for other sports such as jai alai and races involving bicycles, motorcycles, motorboats, and greyhounds (\cite{baron2007parimutuel}). Certain prizes for major lotteries such as Mega Millions, Powerball, and ``EuroMillions'' are computed in a parimutuel fashion (\cite{euromil}). Parimutuel systems are increasingly popular methods for distributing payoffs in online prediction markets as well (\cite{Peters2007}). Goldman Sachs Group, Inc., Deutsche Bank AG, CME Group Inc., Deutsche B\"{o}rse AG, and ICAP have even facilitated the development of parimutuel derivatives on economic indicators (\cite{baron2007parimutuel}).

The first scholarly publication on parimutuel wagering was written by Borel in 1938 (\cite{borel1938pari}), and more recent surveys and anthologies (\cite{thaler1988parimutuel}; \cite{hausch2011handbook}; \cite{hausch2008efficiency}) attest to the substantial academic interest garnered by this topic since. A vast range of issues from optimal betting (\cite{isaacs}; \cite{rosner}; \cite{bolt+chap}) to market efficiency (\cite{hzr};  \cite{ASCH1982187}; \cite{hurley1995note}) to market microstructure (\cite{lange+econ}; \cite{Peters2007}; \cite{Pennock}) has been extensively studied. Significant attention has been paid to strategic interactions among bettors (\cite{quandt+bet+eq}; \cite{ott+sor+surprise}; \cite{Koessler2008733}; \cite{terr+farmer}; \cite{plott}).

Similar to the rise of high-frequency and algorithmic traders in financial markets, a growing number of parimutuel wagering event participants are organizations employing large-scale strategies based upon advanced mathematical, statistical, and computational techniques (\cite{wired}). There are fundamental differences between these bettors and more traditional wagerers. The new firms typically have access to vast budgets, making their betting totals orders of magnitude beyond the amounts wagered by regular players. Often, they can place their wagers at speeds impossible for ordinary bettors to match. Presumably, their use of complex methods also makes their forecasts and corresponding wagering strategies generally superior.

The house collects a percentage of the total amount wagered and, therefore, may initially benefit from the presence of large-scale wagering firms. After all, their activities should increase the size of the pool, at first anyway. The factors just described are thought to put ordinary bettors at an extreme disadvantage, though. Since payouts are calculated according to (\ref{friction pmg payoff words}), ordinary bettors' profits may even directly decline as a result of the large-scale firms' wagers. If this discourages enough regular players from betting, then pool sizes may eventually dwindle, hurting the house's revenue. In fact, many betting organizers have publicly expressed strong concerns about the new breed of wagerers. Betting organizers have even occasionally banned these participants from parimutuel wagering events (\cite{wired}).

How reasonable is this narrative? Our goal is to quantify the impact of large-scale participants in parimutuel wagering events on the house and ordinary bettors.

First, using the theory of {\it large generalized games}, i.e., games with a continuum of diffuse (or non-atomic/minor) players and finitely many atomic (or major) players, we develop a model of parimutuel betting. The bets made by individual atomic players affect all others because they change the final payoff per unit bet on Outcome $i$:
\begin{align}\label{payouts per unit bet words}
\kappa  \left( \frac{ \text{Total Amount Wagered}}{\text{Total Amount Wagered on Outcome $i$}} \right)  . 
\end{align}
Aggregate decisions made by the diffuse players also affect every player for the same reason. A key feature is that an individual diffuse player cannot change (\ref{payouts per unit bet words}) by revising her wagers. In fact, her specific choices have no effect whatsoever on the rest of the game's participants.

We view diffuse and atomic players as stand-ins for ordinary bettors and large-scale wagering firms, respectively. The approximation is motivated by the observation that the total amount wagered by a single traditional bettor is generally negligible compared to the total amount wagered by an entire betting firm.

Our main theoretical result, Theorem \ref{main thm}, provides necessary and sufficient conditions for the existence and uniqueness of a pure-strategy Nash equilibrium. Other scholars have shown the existence of equilibria in a broad class of large generalized games (\cite{Balder1999207}; \cite{Balder2002437}; \cite{Carmona2014130}; \cite{riascos2013pure}). Such results often rely upon sophisticated technology including variants of the Kakutani fixed-point theorem. We choose to employ a more elementary fixed-point argument instead. Advantages of our approach include its simplicity and the possibility of proving the equilibrium's uniqueness. A simple algorithm for computing relevant equilibrium quantities immediately presents itself as well.

Having such an algorithm allows us to analyze our problem in specific scenarios. In accordance with the prevailing narrative, we observe that because of the atomic player, the house is temporarily better off and the diffuse players are worse off in Example \ref{exmp 1}. For varying reasons, at least one of these effects is not observed in each of the remaining situations. 

In Example \ref{exmp 2}, the diffuse players are better off in the presence of the atomic player. Intuitively, when the event is {\it too close to call}, the atomic player can bet on the {\it wrong} outcome even if her prediction is assumed to be quite accurate. Such an error is to the advantage of the diffuse players.

In Example \ref{exmp 3}, the diffuse players {\it believe} that they are better off when there is an atomic player. Roughly, if the diffuse players' beliefs are too homogeneous but the atomic player disagrees with them, the diffuse players' expected profits per unit bet rise when the atomic player takes the other side of their wagers.

In Example \ref{exmp 4}, we argue that the diffuse players are better off but the house is (immediately) worse off because of the atomic player, exactly the opposite of the prevailing narrative.\footnote{Recall that the alleged decline in the house's revenue is thought to occur over time.} To make this point, we recast our model as a two-stage game taking into account the house's strategic decisions. Effectively, because the atomic player considers her impact on (\ref{payouts per unit bet words}), she has a lower tolerance for unfavorable betting conditions than the diffuse player. When she is absent, this means that the house can more easily prey upon the diffuse players.

Before offering further details, we more thoroughly discuss related literature in Section \ref{lit review sect}. We carefully present our model in Section \ref{model sect} and our main theoretical result in Section \ref{main result sect}. We numerically investigate our concrete examples in Section \ref{leb meas sect}. Appendix \ref{corr app} highlights a few technical aspects of Section \ref{main result sect}'s work. We give our longer formal proofs in Appendices \ref{prop 2 proof app} and \ref{main thm proof app}.

\section{Related Literature}\label{lit review sect}

The individual states of our diffuse players are only coupled via the empirical distribution of controls. Since each diffuse player is too small to influence this distribution, she treats it as fixed when determining her own strategy. Assumptions like these have appeared in the literature on continuum games (\cite{aumann}; \cite{schmeidler}; \cite{MASCOLELL}; \cite{rath}) and mean-field games (\cite{Lasry2007}; \cite{huang2006}).\footnote{Originally, players were coupled via the empirical distribution of states, not controls, in the mean-field game literature. Recent advances have shown that models with additional interactions through the controls also have promising applications, say in the contexts of price impact, optimal execution, high-frequency trading, 
and oligopolistic energy market problems
(\cite{2013arXiv1305.2600G}; \cite{carmona2015}; \cite{2015arXiv150807914G}; \cite{chan+sirc+frack}; 
\cite{chan+sirc+courn}).}

Our paper bears a stronger resemblance to work in the former category. For instance, we model parimutuel wagering as a static game. Such a choice is quite common in a continuum game study; however, stochastic differential games are more often the focus in mean-field game theory. Also, we restrict ourselves to an intuitive argument for viewing ordinary bettors as diffuse players. Papers on mean-field games often rigorously present their continuum model as a limit of finite population models, while those on continuum games typically emphasize other issues.

General mean-field interactions among players can be described by complex functions of the empirical distributions of states and/or controls. On the other hand, parimutuel wagerers affect one another through (\ref{payouts per unit bet words}) alone, a comparatively simple scenario. This makes parimutuel wagering an especially strong candidate for modeling by either theory. Continuum games have already been applied in this way (\cite{ottaviani2006timing}; \cite{Takahiro199785}). 

Watanabe considered a two-stage game with a betting organizer and a continuum of risk-neutral diffuse players with heterogeneous beliefs (\cite{Takahiro199785}). First, the betting organizer selects a value for the house take. The diffuse players then determine whether to place a unit bet on one of two outcomes or bet nothing at all. Using techniques from set-valued analysis, Watanabe showed that an equilibrium always exists, provided the house take is not too large. As long as each player can only bet negligible amounts, Watanabe also found that equilibria in parimutuel wagering games are {\it regular}. That is, if a player wagering on Outcome $i$ believes that Outcome $i$ will occur with probability $p$, all players who believe that Outcome $i$ will occur with probability $p^{\prime} > p$ also wager on Outcome $i$. The paper's results on the betting organizer's optimal strategy were in the context of specific examples.

Ottaviani and S\o renson developed their continuum game to explain two phenomena frequently observed in the context of parimutuel wagering on horse races: {\it late informed betting} and the {\it favorite-longshot bias} (\cite{ottaviani2006timing}). The first states that more accurate information about a race's outcome can be gleaned from late bets than early bets. The second says that the public tends to excessively bet on unlikely outcomes and wager too little on likely outcomes. A continuum of privately informed risk-neutral players decide when to place their individual bets in a discrete-time setting. They can wager a unit amount on one of two outcomes or abstain from betting. The paper gave conditions under which all of these players simultaneously wager at the terminal time, and the corresponding equilibrium is shown to always feature the favorite-longshot bias.

Others have more implicitly created infinite-player parimutuel wagering models by assuming that there are {\it many} bettors (\cite{hurley1995note}; \cite{Ottaviani08thefavorite-longshot}). These references have sought to identify other possible causes of the favorite-longshot bias.

The most important new feature of our setup is that, in addition to the diffuse bettors, we introduce an atomic bettor. The parimutuel wagering studies we just discussed only incorporate diffuse players. We would not be able to understand the effects of large-scale wagering organizations on ordinary bettors, if we made a similar assumption.

Because of this addition, our model belongs to the class of continuum game models known as {\it large generalized games}. Games that include both diffuse and atomic players can be found in mean-field game theory as well under the heading {\it major-minor player models} (\cite{huang}; \cite{nguyen+huang}; \cite{nourian+caines}; \cite{jaimungal2015mean}; \cite{wang+tang+huang}). Applications of large generalized games are known to be diverse and already include a collection of problems from politics (\cite{Correa2014}) to oligopolistic markets (\cite{wiszniewska2008dynamic}). General results on the existence of equilibria in large generalized games have also been obtained (\cite{Balder1999207}; \cite{Balder2002437}; \cite{Carmona2014130}; \cite{riascos2013pure}). We choose not to rely upon these, as the simple structural aspects of parimutuel wagering just discussed, combined with a convenient modeling assumption (see Section \ref{model  sect}), allow us to use elementary arguments.

Ultimately, we employ a mean-field approximation for the standard reason: By doing so, we make our model tractable, hopefully while preserving the critical macroscopic properties of our original problem. Issues other than the impact of large-scale wagering organizations have been resolved in finite-player settings, though such studies have usually invoked other strong assumptions.

Weber gave sufficient conditions for the existence of an equilibrium in a simultaneous parimutuel betting game with $N$ atomic players, each of whom is risk-neutral and must bet a specific total amount (\cite{MR637486}). Watanabe, Nonoyama, and Mori's setup and goals are similar to those in Watanabe's continuum game paper, except that the diffuse players in the latter are replaced by finitely many atomic players (\cite{wat+non+mori}; \cite{Takahiro199785}). Explanations of the favorite-longshot bias have been offered using equilibrium results for $N$-player parimutuel wagering games (\cite{chadha1996betting}; \cite{Koessler2008733}; \cite{ott+sor+surprise}; \cite{potters+wit}). Games in which finitely many bettors wager sequentially have also been investigated for various purposes (\cite{Feeney2001165}; \cite{cheung}; \cite{koess+zieg+broi}; \cite{Koessler2008733}; \cite{thrall}). For example, Thrall showed that if risk-neutral atomic bettors with homogeneous beliefs wager sequentially, their profits tend to zero as the number of bettors increases (\cite{thrall}). Note that some of these works do consider limiting cases in which the population of wagerers grows arbitrarily large to complement their other insights (\cite{thrall}; \cite{ott+sor+surprise}).

\section{Model Details}\label{model sect}

Our players have the opportunity to wager on an event that can unfold in two mutually exclusive ways: Outcome 1 might occur. If not, Outcome 2 will occur. For now, we view $\kappa \in \left( 0 , 1 \right)$ to be exogenously given. Inspired by Watanabe et al. (\cite{wat+non+mori}; \cite{Takahiro199785}), we later informally consider what happens when we allow the house take to be optimally selected by the betting organizer in the first stage of a two-stage game (see Example \ref{exmp 4}). Our results in Section \ref{main result sect} are unaffected by such a shift.

The unit interval describes the diffuse bettors' views on the likelihood of Outcome 1: the bettors whose views are indexed by $p \in \left[ 0 , 1 \right]$ believes that Outcome 1 will occur with probability $p$. Initially, each diffuse bettor has some (negligible) unit wealth. A finite Borel measure $\mu$ with a continuous everywhere positive density characterizes the distribution of the diffuse bettors. More precisely, the total initial wealth of all diffuse bettors whose views are contained in a Borel set $A$ is $\mu \left( A\right)$.

That $\mu$ has a continuous everywhere positive density is our {\it convenient modeling assumption} from Section \ref{lit review sect}. It is similar to a key hypothesis in Ottaviani and S\o renson's work, although the posterior beliefs of their diffuse bettors are obtained after updating a common prior belief using a private signal and Bayes' rule (\cite{ottaviani2006timing}). Effectively, Watanabe assumed that $\mu$ need not have a density, and even when it does, the density need not be positive everywhere (\cite{Takahiro199785}). These choices necessitated a set-valued approach, which we are able to avoid. 

Continuity merely simplifies a few of our arguments, e.g., see Step \ref{step 6 alt varphi dec} of Theorem \ref{main thm}'s proof. The rest of the assumption plays a more critical role. Results on the existence of equilibria in parimutuel wagering games often include a hypothesis such as the following: for any given outcome, at least two bettors believe that the outcome will occur with positive probability.\footnote{This is true of Weber's work, for instance (\cite{MR637486}). Roughly, if only one player believes that a particular outcome will occur with positive probability, she should wager an arbitrarily small amount on that outcome. In the absence of a positive minimum bet size constraint, it follows that an equilibrium does not exist.} By supposing that the density is positive, we assume this as well. Watanabe has shown that an equilibrium may not be unique, if $\mu \left( \left\{ p \right\} \right) > 0$  for some fixed $p$ (\cite{Takahiro199785}). Obviously, this situation is prevented by the density's existence.

Our atomic player believes that Outcome 1 will occur with probability $q \in \left[ 0 , 1 \right]$. She has (non-negligible) finite initial wealth $w > 0$.

Throughout, we treat all players' beliefs as exogenously determined. We do not address how the players generate their estimates; however, this process is of great interest both practically and academically (\cite{hausch2011handbook}; \cite{hausch2008efficiency}). For our theoretical results in Section \ref{main result sect}, we also do not specify how the players' estimates compare to the {\it actual} probability that Outcome 1 will occur. Since large-scale betting organization allegedly produce highly accurate forecasts, we could choose $q$ to be some small perturbation of the actual probability of Outcome 1. We informally experiment with this extra assumption in Examples \ref{exmp 1} and \ref{exmp 2}.

All players decide how much to wager on each outcome. Their choices are constrained only by their initial wealth: a betting strategy is {\it feasible} (or {\it admissible}) for an individual bettor as long as the sum of her wagers is no more than her wealth. For example, a bettor could choose to wager 100\% of her wealth on Outcome 1, 55\% of her wealth on Outcome 1 and 30\% of her wealth on Outcome 2, or not wager at all. We formalize this as follows.\footnote{In practice, concerns about large-scale wagering firms are also driven by their alleged ability to bet {\it faster} than ordinary players. Our static game only has two outcomes, so it does not make sense to model this feature here. We hope to revisit the issue in a future work.}

\begin{defn}\label{strat profile defn}
A {\it feasible strategy profile for the diffuse players} is a measurable function 
\begin{equation*}
f = \left( f_1 , f_2 \right) : \left( 0 , 1 \right)  \longrightarrow \left\{ \left(x_1 , x_2 \right) \in \mathbb{R}^2_{\geq 0} \,  : \,  x_1 + x_2 \leq 1 \right\}.  
\end{equation*}
A {\it feasible strategy profile for the atomic player} is a vector $a = \left( a_1 , a_2 \right) \in \mathbb{R}^2_{\geq 0}$ such that 
\begin{equation*}
a_1 + a_2 \leq w. 
\end{equation*}
We call the pair $\left( f , a \right)$ a {\it feasible strategy profile}. 
\end{defn}

Under $\left( f , a \right)$, the atomic player wagers $a_i$ on Outcome $i$. Each diffuse player who believes that Outcome 1 will occur with probability $p$ wagers $f_i \left( p \right) \times 100 \%$ of her (negligible) unit initial wealth on Outcome $i$.

Our space of feasible strategy profiles is slightly atypical. Previously, diffuse players in parimutuel wagering games have only been able to place unit bets, if they bet at all (\cite{ottaviani2006timing}; \cite{Takahiro199785}). We could have made this restriction as well without loss of generality due to Proposition \ref{diff equilib bets prop}. Watanabe allowed groups of diffuse players to wager differently, even if they held identical beliefs (\cite{Takahiro199785}). In equilibrium, such a discrepancy could only arise among the diffuse players who believed that their expected profits would be zero. We encounter a related ambiguity in our framework (see our discussion of Proposition \ref{diff equilib bets prop}). For us, the $\mu$-measure of this set of bettors is zero, and we anticipate that all of our main results would remain the same, if we were to relax the diffuse bettors' {\it same beliefs-same bets} restriction. More significantly, atomic players have been frequently constrained to wager a fixed amount in total or a unit amount on a single outcome when they bet (\cite{MR637486}; \cite{wat+non+mori}; \cite{chadha1996betting}; \cite{ott+sor+surprise}).\footnote{An exception is Cheung's thesis, though that work's focus is quite different from our own (\cite{cheung}).} Proposition \ref{atom equilib bets prop} suggests that imposing these restrictions would have a severe effect.

The total amount that the diffuse players wager on Outcome $i$, denoted $d_i$, is given by
\begin{equation*}
d_i = \displaystyle\int_{0}^1 f_i\left( p \right) \mu \left( d p \right) .  
\end{equation*}
Since $\mu$ has a density, we immediately confirm that the bets placed by any given diffuse player are too small to affect the amount wagered on any specific outcome. Of course, if she revises her strategy, then a particular diffuse player affects neither the total amount wagered nor (\ref{payouts per unit bet words}). All of these quantities could change when aggregations of diffuse players, that is, collections of diffuse players whose beliefs are contained in a Borel set $A$ with positive $\mu$-measure, revise their wagers.

Payoffs are determined according to (\ref{friction pmg payoff words}). Each player selects her wagering strategy simultaneously in order to maximize her expected profit according to her beliefs. We implicitly assume that every bettor knows $\kappa$, $\mu$, $q$, and $w$, so that she can select the best response to her opponents' collective actions. Similar assumptions can be found in many other equilibrium studies on parimutuel wagering (\cite{MR637486}; \cite{wat+non+mori}; \cite{Takahiro199785}; \cite{chadha1996betting}; \cite{ott+sor+surprise}).\footnote{Even so, our players might seem unrealistically knowledgeable and confident in their beliefs.
Real bettors would presumably have a more complex prior for the outcomes' likelihoods. 
They probably would not have access to such comprehensive information about their opponents.
If they somehow had a sense of these details, they might also wish to update their 
own forecasts.

An ad-hoc but simple way to address these objections could be to consider 
$\kappa$ as some (publicly known) negative perturbation of the true $\kappa$ 
(abusing notation).
The idea is that each player would model transaction costs as being higher 
than their actual value, artificially lowering their perceived edge and 
encouraging them to bet more cautiously than they would otherwise.

A more thorough study could begin by determining whether or not
approximating parimutuel wagering using the Nash equilibrium solution concept is, indeed,
reasonable.
If it is, one might endow the bettors with more sophisticated priors
and enable them to update their beliefs using the equilibrium 
implied probabilities (see Definition \ref{impl prob defn}).
This would lead to an extra condition in Definition \ref{maj min NE def}. 

If the solution concept is unreasonable, one could devise 
a new scenario in which players independently determine their betting
strategies according to individual reference models and appropriately penalized 
alternative models for outcome likelihoods, as well as their opponents' parameters.
Broadly speaking, this treatment of a single player's optimization problem has seen widespread 
use across macroeconomics and finance (\cite{herr+jmk+hedg}).
One might then investigate what unfolds when all players simultaneously participate
in the same parimutuel wagering event.

We leave further consideration of these topics for a future work.}

Since each diffuse player starts out with negligible unit wealth, technically, we should only discuss the expected profits of diffuse bettors whose views lie in some Borel set $A$. We nevertheless compute and refer to the expected profits of an individual diffuse bettor in an obvious, but admittedly informal, way. Doing so helps motivate our definition of a {\it pure-strategy Nash equilibrium} (see Definition \ref{maj min NE def}) and highlight the intuition underlying our results.

A seemingly more formidable concern is how to handle (\ref{friction pmg payoff words}) in pathological cases. It is trivial to produce a feasible strategy profile $\left( f, a \right)$ such that for some $p$, we have $f_j \left( p \right) > 0$ but 
\begin{equation*}
d_j = a_j  = 0 . 
\end{equation*}
Na{\"i}vely translating (\ref{friction pmg payoff words}), we conclude that a diffuse player whose views are indexed by $p$ receives
\begin{align}\label{diff pay eqn bef impl prob}
\kappa \left( \displaystyle\sum_{i=1 }^2 \left(  d_i  + a_i  \right) \right) \left(  \displaystyle\frac{  f_j \left( p \right) }{    d_j + a_j  } \right)
\end{align}
whenever Outcome $j$ occurs.

If
\begin{equation*}
d_1 = d_2 = a_1= a_2  = 0 ,
\end{equation*}
then the total amount wagered is zero. Practically, the betting organizer would probably cancel such an event, which suggests that it is natural to set all players' payoffs to zero in this scenario.

Alternatively, we might have 
\begin{equation*}
d_i+ a_i  >0  
\end{equation*}
for $i \ne j$. There are now two appealing options for the diffuse player's payoff. First, we might choose to set the payoff to zero. Practically, no bettor would receive a payout, if Outcome $j$ occurred but no one wagered on it. We also might set the payoff to $+\infty$, as in Watanabe's work (\cite{Takahiro199785}). In practice, if the amount wagered on Outcome $j$ were zero, each player who believed that Outcome $j$ would occur with positive probability would want to place an arbitrarily small bet on Outcome $j$. Setting the payoff to $+\infty$ captures this intuition.

We choose the first option, but selecting the second instead would not change our results. Only equilibrium payoffs need to be computed, and in an equilibrium, positive amounts are always wagered on both outcomes. Essentially, the scenario we describe never arises. One reason is that we do not allow the {\it trivial} (or {\it null}) equilibrium in which no player wagers. We could,\footnote{This situation could be considered an equilibrium because if any single player unilaterally revised her wagers, intuitively, she should incur a loss of at least $\left( 1 - \kappa \right)\%$.} but as Watanabe observed, that case is comparatively uninteresting and practically unimportant (\cite{Takahiro199785}). Roughly, the other reason is the same as our justification for possibly setting the payoff to $+\infty$.

Before making this discussion precise, we introduce some notation.

\begin{defn}\label{impl prob defn}
Given a feasible strategy profile $\left( f , a \right)$ such that at least one of the $d_i$'s or $a_i$'s is positive, the {\it implied (or subjective) probability} that Outcome 1 will occur, denoted $P^{f,a}$, is defined by 
\begin{equation*}
P^{f,a} = \displaystyle\frac{ d_1  + a_1 }{  \sum_{i=1 }^2 \left(  d_i   + a_i  \right) }   . 
\end{equation*}
We refer to  
\begin{equation*}
1- P^{f,a} = \displaystyle\frac{ d_2  + a_2 }{  \sum_{i=1 }^2 \left(  d_i   + a_i  \right) }    
\end{equation*}
as the {\it implied (or subjective) probability} that Outcome 2 will occur. 
\end{defn}

$P^{f,a}$ is the ratio of the amount wagered on Outcome 1 to the total amount wagered, assuming the latter is positive. Our previous discussion implies that $P^{f,a} \in \left( 0 , 1 \right)$ in equilibrium.\footnote{We later observe that $P^{f,a} \in \left(  1- \kappa , \kappa \right)$ in equilibrium (see Step \ref{step 5 orig} of Theorem \ref{main thm}'s proof).} Since the argument was informal, we do not yet take this as fact. In particular, the amount received by a diffuse player who believes that Outcome 1 will occur with probability $p$ is\footnote{Of course, we still abuse notation here. When one of our indicator functions is equal to zero, the corresponding fraction is actually of the form $0/0$, not zero as we suppose.}
\begin{align*}
\kappa \left( \displaystyle\sum_{i=1 }^2 \left(  d_i  + a_i  \right) \right) \left(  \displaystyle\frac{  f_1 \left( p \right) \mathds{1}_{ \left\{ d_1 + a_1 \ne 0 \right\} }   }{    d_1 + a_1  } \right)  =  \displaystyle\frac{ \kappa  f_1 \left( p \right) \mathds{1}_{ \left\{ P^{f,a} \ne 0 \right\} }  }{  P^{f,a}   } 
\end{align*}
and
\begin{align*}
\kappa \left( \displaystyle\sum_{i=1 }^2 \left(  d_i  + a_i  \right) \right) \left(  \displaystyle\frac{  f_2 \left( p \right) \mathds{1}_{\left\{ d_2 + a_2 \ne 0 \right\}  }   }{    d_2 + a_2  } \right)  = \displaystyle\frac{ \kappa  f_2 \left( p \right) \mathds{1}_{ \left\{ P^{f,a} \ne 1 \right\} }  }{  1- P^{f,a}   } 
\end{align*}
when Outcomes 1 and 2 occur, respectively. Hence, this diffuse player believes that her expected profit is
\begin{align}\label{exp prof diff player intuit}
f_1 \left( p \right)  \left( \displaystyle\frac{ \kappa \, \mathds{1}_{\left\{ P^{f, a}\ne 0 \right\}} \,  p }{ P^{f, a}} - 1 \right) + f_2 \left( p \right)  \left( \displaystyle\frac{ \kappa \, \mathds{1}_{ \left\{  P^{f, a} \ne 1  \right\} } \,  \left( 1 - p \right) }{  1 - P^{f, a} } - 1 \right) . 
\end{align}
Similarly, the atomic player thinks that her expected profit is 
\begin{align}\label{exp prof atom player intuit}
a_1  \left( \displaystyle\frac{ \kappa \, \mathds{1}_{\left\{ P^{f, a}\ne 0 \right\}} \,  q }{ P^{f, a}} - 1 \right) + a_2 \left( \displaystyle\frac{ \kappa \, \mathds{1}_{ \left\{  P^{f, a} \ne 1  \right\} } \,  \left( 1 - q \right) }{  1 - P^{f, a} } - 1 \right)  . 
\end{align}

\begin{defn}\label{maj min NE def}
A {\it pure-strategy Nash equilibrium} is a feasible strategy profile $\left( f^\star , a^\star \right)$ such that 
\begin{enumerate}[label=(\roman*)]
\item \label{NE def orig i} at least one of the $d_i^{\star}$'s or $a_i^{\star}$'s is positive;
\item \label{NE def orig ii} for any $p \in \left[ 0 , 1 \right]$, 
\begin{align*}
&f_1^{\star} \left( p \right)  \left( \displaystyle\frac{ \kappa \, \mathds{1}_{\left\{ P^{f^{\star}, a^{\star}}\ne 0 \right\}} \,  p }{ P^{f^{\star}, a^{\star}}} - 1 \right) + f_2^{\star} \left( p \right)  \left( \displaystyle\frac{ \kappa \, \mathds{1}_{ \left\{  P^{f^{\star}, a^{\star}} \ne 1  \right\} } \,  \left( 1 - p \right) }{  1 - P^{f^{\star}, a^{\star}} } - 1 \right) \notag \\
&\quad = \displaystyle\sup_{\substack{   b_1 , b_2 \geq 0 \\ b_1 + b_2 \leq 1 }} \left\{ b_1   \left( \displaystyle\frac{ \kappa \, \mathds{1}_{\left\{ P^{f^{\star}, a^{\star}}\ne 0 \right\}} \,  p }{ P^{f^{\star}, a^{\star}}} - 1 \right) + b_2  \left( \displaystyle\frac{ \kappa \, \mathds{1}_{ \left\{  P^{f^{\star}, a^{\star}} \ne 1  \right\} } \,  \left( 1 - p \right) }{1 - P^{f^{\star}, a^{\star}} } - 1 \right)  \right\} ;
\end{align*}
\item \label{NE def orig iii} and
\begin{align*}
& a_1^{\star}  \left( \displaystyle\frac{ \kappa \, \mathds{1}_{\left\{ P^{f^{\star}, a^{\star}}\ne 0 \right\}} \,  q }{ P^{f^{\star}, a^{\star}}} - 1 \right) + a_2^{\star} \left( \displaystyle\frac{ \kappa \, \mathds{1}_{ \left\{  P^{f^{\star}, a^{\star}} \ne 1  \right\} } \,  \left( 1 - q\right) }{  1 - P^{f^{\star}, a^{\star}} } - 1 \right) \notag \\
&\quad = \displaystyle\sup_{\substack{   b_1 , b_2 \geq 0 \\ b_1 + b_2 \leq w \\ d_1^{\star}, d_2^{\star}, b_1 \text{ or } b_2 > 0}} \left\{ b_1   \left( \displaystyle\frac{ \kappa \, \mathds{1}_{\left\{ P^{f^{\star}, b }\ne 0 \right\}} \,  q }{ P^{f^{\star}, b }} - 1 \right) + b_2  \left( \displaystyle\frac{ \kappa \, \mathds{1}_{ \left\{  P^{f^{\star}, b } \ne 1  \right\} } \,  \left( 1 -q \right) }{1 - P^{f^{\star}, b} } - 1 \right)  \right\} . 
\end{align*}
\end{enumerate}
\end{defn}

\ref{NE def orig i} formally excludes the case in which the total amount wagered is zero. \ref{NE def orig ii} and \ref{NE def orig iii} ensure that each player maximizes her expected profit according to her beliefs, given her opponents' wagers.\footnote{In \ref{NE def orig ii}, we determine the equilibrium wagers for all diffuse bettors whose beliefs are indexed by $p$; however, a single diffuse bettor with these beliefs solves the same maximization problem.}

First, observe that each player requires very little information about her opponents' strategies. For a given diffuse bettor, knowing $P^{f^{\star}, a^{\star}}$ alone is enough. The atomic player must be able to compute $P^{f^{\star}, b}$ for all of her feasible strategy profiles $b$, so it is sufficient for her to know $d_1^{\star}$ and $d_2^{\star}$. The difference for the two kinds of players reflects that an individual diffuse player cannot affect the implied probability of Outcome 1, while the atomic player can. These remarks explain why we informally claim that only the atomic player and aggregations of diffuse players affect the other participants. Notice that each player's strategy depends {\it anonymously} on her opponents' bets:  how, specifically, her opponents wagers produced $P^{f^{\star}, a^{\star}}$, $d_1^{\star}$, and $d_2^{\star}$ is irrelevant.

Although we only optimize over $b$ in \ref{NE def orig iii}, the apparent possibility that $d_1^{\star} =  d_2^{\star} =0$ motivates our use of the extra constraint 
\begin{equation*}
d_1^{\star}, d_2^{\star}, b_1 \text{ or } b_2 > 0 .
\end{equation*}
One might be concerned that we do not consider the feasible strategy $b_1 = b_2 = 0$ for the atomic player, if $d_1^{\star} =  d_2^{\star} =0$. Recall that all players receive a payoff of zero in such a situation. A simple calculation shows that the supremum is then also zero, so no issue is caused by our omission.

The last important concept for our modeling framework is uniqueness.

\begin{defn}\label{uniq equilib def}
A pure-strategy Nash equilibrium $\left( f^\star , a^\star \right)$ is {\it unique} if for any other pure-strategy Nash equilibrium $\left( f^\diamond , a^\diamond \right)$, we have $f^{\star} = f^{\diamond}$ $\mu$-a.s. and $a^{\star}=a^{\diamond}$. 
\end{defn}

We allow $f^{\star}$ and $f^{\diamond}$ to disagree on a set of $\mu$-measure zero since, ultimately, all relevant equilibrium quantities such as the implied probabilities are unaffected by such a difference. Soon, we see that $f^{\star} \left( p \right)$ and $f^{\diamond} \left( p \right)$ must be equal for all but two points, at most. There is only uncertainty about the behavior of the diffuse bettors who believe that their expected profits are zero (cf. our discussion about our space of feasible strategy profiles).

\section{Theoretical Results}\label{main result sect}

We now state and prove\footnote{We describe the ideas underlying all of our proofs in Section \ref{main result sect}; however, we delay our formal arguments for Proposition \ref{atom equilib bets prop} and Theorem \ref{main thm} until Appendices \ref{prop 2 proof app} and \ref{main thm proof app}, respectively.} our theoretical results, beginning with Propositions \ref{diff equilib bets prop} and \ref{atom equilib bets prop}. The former describes how the diffuse players should wager in response to the atomic player's strategy. The latter tells us how the atomic player should bet, given the diffuse players' wagers. We use these observations to prove Theorem \ref{main thm}, our main result.\footnote{Roughly, this suggests that our large generalized game can almost be viewed as a game with two players: the mean-field of diffuse players and the atomic player.} Recall that it offers necessary and sufficient conditions for the existence and uniqueness of a pure-strategy Nash equilibrium.
 
We conclude Section \ref{main result sect} with Corollaries \ref{atom pl exp prof pos cor}, \ref{reg equilib cor}, and \ref{impl prob kappa  0.5}. Corollary \ref{atom pl exp prof pos cor} says that the atomic player wagers on a particular outcome if and only if the final expected profit per unit bet on that outcome is positive. The next corollary states that our equilibria are regular in a particular sense, while Corollary \ref{impl prob kappa  0.5} says that the implied probability of Outcome 1 tends to $ 0.5$ uniformly as the house take approaches $50 \%$.

\begin{prop}\label{diff equilib bets prop}
Let $\left( f , a \right)$ be a feasible strategy profile such that $d_1$, $d_2 > 0$. $f$ satisfies 
\begin{align}\label{max problem prop 1}
&f_1  \left( p \right)  \left( \displaystyle\frac{ \kappa \,  p }{ P^{f , a }} - 1 \right) + f_2  \left( p \right)  \left( \displaystyle\frac{ \kappa \,  \left( 1 - p \right) }{  1 - P^{f , a } } - 1 \right) \notag \\
&\quad = \displaystyle\sup_{\substack{   b_1 , b_2 \geq 0 \\ b_1 + b_2 \leq 1 }} \left\{ b_1   \left( \displaystyle\frac{ \kappa  \,  p }{ P^{f , a }} - 1 \right) + b_2  \left( \displaystyle\frac{ \kappa \,  \left( 1 - p \right) }{1 - P^{f , a } } - 1 \right)  \right\} 
\end{align}
for all $p \in \left[ 0 , 1 \right]$ if and only if 
\begin{align*}
f_1 \left( p \right) &= \left\{ \begin{array}{cc}
1 & \quad \text{if } \, \,  p > P^{f,a}  / \kappa \\
& \\
0 & \quad \text{if } \, \, p < P^{f,a}  / \kappa \\
\end{array}   \right. \qquad \qquad
f_2 \left( p \right) = \left\{ \begin{array}{cc}
1 & \quad \text{if } \, \, 1 - p  > \left( 1 -  P^{f,a} \right) / \kappa  \\
&\\
0 & \quad \text{if } \, \, 1 - p  < \left( 1 -  P^{f,a} \right) / \kappa \\
\end{array} \right. .
\end{align*}
\end{prop}

First, notice that $ P^{f,a} \in \left( 0 ,1 \right)$ since both $d_1$ and $d_2$ are positive. Comparing (\ref{max problem prop 1}) and \ref{NE def orig ii} of Definition \ref{maj min NE def}, we find that Proposition \ref{diff equilib bets prop} provides a simple characterization of the diffuse players' equilibrium strategies in this case, given the bets of the atomic player. Our assumption is not too restrictive, as Step \ref{step 1 orig} of Theorem \ref{main thm}'s proof says that $d_1$ and $d_2$ are always both positive in equilibrium.

Three distinct groups of diffuse bettors emerge: The first group, containing the diffuse bettors who believe that Outcome 1 will occur with probability greater than $P^{f, a}/ \kappa$, wager all of their initial wealth on Outcome 1. Diffuse bettors who think that Outcome 2 will occur with probability greater than $\left( 1 - P^{f, a}\right) / \kappa$ make up the second group. These players bet their entire fortunes on Outcome 2. The remaining diffuse players, except those whose beliefs are indexed by $p = P^{f, a}/ \kappa$ or $\left( 1 - P^{f, a}\right) / \kappa$, do not wager at all. None of the groups overlap, since $0 < \kappa < 1$ implies that 
\begin{equation}\label{fi star nonoverlap eqn}
  1 - \left( \displaystyle\frac{ 1 -  P^{f, a}}{ \kappa} \right) <  \displaystyle\frac{P^{f, a}}{\kappa} . 
\end{equation}

The proof's underlying intuition is easy to explain. It is equivalent to show that $f$ satisfies (\ref{max problem prop 1}) on $\left[ 0 , 1 \right]$ if and only if
\begin{align*}
f_1 \left( p \right) &= \left\{ \begin{array}{cc}
1 & \quad \text{if } \, \,  \frac{ \kappa   p }{ P^{f, a}} - 1 > 0\\
& \\
0 & \quad \text{if } \, \, \frac{ \kappa   p }{ P^{f, a}} - 1 < 0\\
\end{array}   \right. \qquad \qquad
f_2 \left( p \right) = \left\{ \begin{array}{cc}
1 & \quad \text{if } \, \,  \frac{ \kappa \left( 1 - p \right) }{  1 - P^{f, a} } - 1 > 0 \\
&\\
0 & \quad \text{if } \, \, \frac{ \kappa \left( 1 - p \right) }{  1 - P^{f, a} } - 1 < 0 \\
\end{array} \right. .
\end{align*}
The terms
\begin{equation*}
\displaystyle\frac{ \kappa   p }{ P^{f, a}} - 1 \qquad \text{and} \qquad  \displaystyle\frac{ \kappa   \left( 1 - p \right) }{  1 - P^{f, a} }   -1
\end{equation*}
describe the expected profit per unit bet on Outcomes 1 and 2, respectively, from the perspective of the diffuse player who believes that Outcome 1 will occur with probability $p$. Individual diffuse players are risk-neutral and do not affect these quantities, so they wager on an outcome only if the corresponding term is positive. For a particular diffuse player, this is true of at most one outcome as $\kappa \in \left( 0 , 1 \right)$. Consequently, if a diffuse player has identified a profitable wagering opportunity, she bets her entire fortune on it.

Despite their large space of feasible strategies, the diffuse players, aside from those whose beliefs are indexed by $p = P^{f, a}/ \kappa$ or $\left( 1 - P^{f, a}\right) / \kappa$, wager either $100\%$ or $0\%$ of their wealth on each outcome. The value of $f_i$ at $P^{f, a} / \kappa$ and $\left( 1 -  P^{f,a} \right) / \kappa$ is ambiguous because, if a given diffuse player's expected profit per unit bet on Outcome $i$ is zero, then she is indifferent to the size of her bet on Outcome $i$.\footnote{Since $\mu$ has a density, the $\mu$-measure of a set with two points is zero. It follows that this ambiguity has no bearing on an equilibrium's uniqueness. A more serious concern is that we could have feasible strategy profiles satisfying Definition \ref{maj min NE def} with different implied probabilities. Precluding this possibility is a key part of Theorem \ref{main thm}'s proof.}

Though their setups differed from our own (see Sections \ref{lit review sect} - \ref{model sect}), Ottaviani, S\o renson, and Watanabe found similar groupings of diffuse players in equilibrium (\cite{ottaviani2006timing}; \cite{Takahiro199785}). We return to this observation during our discussion of Corollary \ref{reg equilib cor}, which roughly says that these groupings persist even when we take into account the atomic player's wagers.

\begin{proof}

There is little to formalize beyond our heuristic discussion above. We only comment that rearranging (\ref{fi star nonoverlap eqn}) shows that we can never have both 
\begin{equation*}
 \displaystyle\frac{ \kappa   p }{ P^{f, a}} - 1  > 0 \qquad \text{and} \qquad \displaystyle\frac{ \kappa  \left( 1 - p \right) }{ 1 - P^{f, a} } - 1  > 0 . 
\end{equation*}

\end{proof}

\begin{prop}\label{atom equilib bets prop}
Let $\left( f , a \right)$ be a feasible strategy profile such that $d_1$, $d_2 > 0$. Consider the equation
\begin{align}\label{atom max prob eqn prop}
& a_1   \bigg( \displaystyle\frac{ \kappa  \,  q }{ P^{f , a }} - 1 \bigg) + a_2  \left( \displaystyle\frac{ \kappa  \,  \left( 1 - q\right) }{  1 - P^{f , a } } - 1 \right) \notag \\
&\quad = \displaystyle\sup_{\substack{   b_1 , b_2 \geq 0 \\ b_1 + b_2 \leq w }} \left\{ b_1   \bigg( \displaystyle\frac{ \kappa  \,  q }{ P^{f , b }} - 1 \bigg) + b_2  \left( \displaystyle\frac{ \kappa \,   \left( 1 -q \right) }{1 - P^{f , b} } - 1 \right)  \right\} . 
\end{align}
\begin{enumerate}[label=(\roman*)]
\item \label{atom bet prop orig i} When
\begin{equation}\label{atom bets out1 ineq alone}
q > \displaystyle\frac{ d_1 }{\kappa \left( d_1 + d_2 \right)} ,
\end{equation}
$a$ satisfies (\ref{atom max prob eqn prop}) if and only if $a_2 = 0$ and 
\begin{align}\label{atom bets out1 soln alone}
 a_1 &= \displaystyle\min \left\{ w ,  \, \displaystyle\sqrt{ \displaystyle\frac{ \kappa  q d_1 d_2}{1 - \kappa q }  } -  d_1  \right\}. 
\end{align}
\item \label{atom bet prop orig ii} When
\begin{equation}\label{atom bets out2 ineq alone}
1 - q > \displaystyle\frac{ d_2 }{\kappa \left( d_1 + d_2 \right)}  , 
\end{equation}
$a$ satisfies (\ref{atom max prob eqn prop})  if and only if $a_1 = 0$ and 
\begin{align}\label{atom bets out2 soln alone}
a_2 &= \displaystyle\min \left\{ w ,  \, \sqrt{ \displaystyle\frac{ \kappa  \left( 1 - q \right) d_1 d_2 }{1 - \kappa \left( 1 - q \right) }  } -  d_2  \right\}. 
\end{align}
\item \label{atom bet prop orig iii} When
\begin{equation}\label{atom bets noth ineq alone}
 q \leq \displaystyle\frac{ d_1 }{\kappa \left( d_1 + d_2 \right)}  \qquad \text{and} \qquad  1 - q \leq \displaystyle\frac{ d_2 }{\kappa \left( d_1 + d_2 \right)}, 
\end{equation}
$a$ satisfies (\ref{atom max prob eqn prop}) if and only if $a_1 = a_2 = 0$.
\end{enumerate}
\end{prop}

As in our discussion of the last result, $P^{f,a} \in \left( 0 ,1 \right)$ and that we only study the case in which $d_1$, $d_2 > 0$ does not matter. Proposition \ref{atom equilib bets prop} can be interpreted as the complement of Proposition \ref{diff equilib bets prop}: It characterizes the atomic player's equilibrium strategy, given the diffuse players' bets.

A calculation similar to that in (\ref{fi star nonoverlap eqn}) shows that (\ref{atom bets out1 ineq alone}) and (\ref{atom bets out2 ineq alone}) never hold simultaneously. Notice that $a_1 > 0$ under \ref{atom bet prop orig i}, while $a_2 > 0$ under \ref{atom bet prop orig ii}. For example, (\ref{atom bets out1 ineq alone}) implies that 
\begin{equation}\label{1 bc impli}
 \displaystyle\frac{ d_1 }{d_1 + d_2 } < \kappa q \qquad \text{and} \qquad  \displaystyle\frac{ d_2 }{d_1 + d_2 } > 1- \kappa q. 
\end{equation}
In this case, 
\begin{equation*}
a_1 = \displaystyle\min \left\{ w ,  \, \displaystyle\sqrt{ \displaystyle\frac{ \kappa  q d_1 d_2}{1 - \kappa q }  } -  d_1  \right\} \geq \displaystyle\min \left\{ w ,  \, \displaystyle\sqrt{ \left( \displaystyle\frac{d_1}{d_2} \right) d_1 d_2   } -  d_1  \right\} > 0. 
\end{equation*}
Hence, under (i), the atomic player wagers on Outcome 1 alone. The atomic player only bets on Outcome 2 in (ii), while she does not wager at all in (iii). Despite having the opportunity to do so, she never simultaneously wagers on both possibilities. Proposition \ref{diff equilib bets prop} revealed similar behavior for diffuse bettors.

Important ideas in the proofs of Propositions \ref{diff equilib bets prop} and \ref{atom equilib bets prop} are closely related. By rearranging (\ref{atom bets out1 ineq alone}) and (\ref{atom bets out2 ineq alone}), we get 
\begin{equation*}
 \displaystyle\frac{ \kappa \left( d_1 + d_2 \right) q }{ d_1}  - 1 > 0 \qquad \text{and} \qquad  \displaystyle\frac{ \kappa \left( d_1 + d_2 \right) \left( 1 - q \right) }{ d_2}  - 1 > 0 , 
\end{equation*}
respectively. Given the wagers of the diffuse players, the first term describes the expected profit per unit bet on Outcome 1 according to the atomic player. The second term has the analogous interpretation for Outcome 2.

As in our analysis for the diffuse players, the atomic player is risk-neutral and bets only when one of these inequalities holds,\footnote{The outcome, if any, on which the atomic player wagers can be identified based upon her opponents' wagers alone. In particular, this identification can be made without knowledge of the implied probability of Outcome 1. Still, from (\ref{atom max prob eqn prop}), it is clear that the atomic player bets on Outcome $i$ in equilibrium if and only if the expected profit per unit bet on Outcome $i$ is positive. We rigorously prove and discuss this further in Corollary \ref{atom pl exp prof pos cor}.} leading directly to \ref{atom bet prop orig iii}. Isaacs first proved this while modeling parimutuel wagering as the control problem faced by a single risk-neutral atomic player unconstrained by a budget (\cite{isaacs}). We include \ref{atom bet prop orig iii} in Proposition \ref{atom equilib bets prop} merely to assist with our presentation.

Unlike our previous analysis, we cannot conclude that the atomic player bets her entire fortune on an initially profitable wagering opportunity. The reason is simple: the atomic player's choices affect (\ref{payouts per unit bet words}). In fact, all else being equal, the payoffs per unit bet on an outcome decrease as the atomic player raises her wager on that outcome. Balancing the desires to increase her expected profit by betting more and keep her expected profit per unit bet high by betting less leads to (\ref{atom bets out1 soln alone}) and (\ref{atom bets out2 soln alone}), not the all-or-nothing wagers of Proposition \ref{diff equilib bets prop}.

More precisely, if (\ref{atom bets out1 ineq alone}) holds and we relax our wealth constraint, this trade-off makes it optimal for the atomic player to wager 
\begin{equation}\label{isaacs wager outcome 1}
\displaystyle\sqrt{ \displaystyle\frac{ \kappa  q d_1 d_2}{1 - \kappa q }  } -  d_1
\end{equation}
on Outcome 1. This solution was also first discovered by Isaacs (\cite{isaacs}).\footnote{Related expressions are also seen in equilibrium studies with $N$ risk-neutral atomic players constrained to wager a specific total amount (\cite{chadha1996betting}). In an obvious way, Proposition \ref{atom equilib bets prop} fills the small gap between these two settings. Recall that no explicit solutions are available for Cheung's atomic player, who faces a budget constraint like ours but is risk-averse (\cite{cheung}).} The idea behind (\ref{atom bets out1 soln alone}) is then clear: If the atomic player cannot afford to bet (\ref{isaacs wager outcome 1}) on Outcome 1, she instead wagers as much as she can. This seems reasonable, intuitively, since up to (\ref{isaacs wager outcome 1}), the positive impact of raising her Outcome 1 bet on her expected profit should outweigh the negative impact. The interpretation of \ref{atom bet prop orig ii} is similar. We present the formal proof of Proposition \ref{atom equilib bets prop} in Appendix \ref{prop 2 proof app}.

\begin{thm}\label{main thm}
A pure-strategy Nash equilibrium exists if and only if $\kappa >  0.5$.\footnote{In practice, this inequality almost always holds.} When an equilibrium exists, it is unique. 
\end{thm}

The connection between low transaction costs (or large $\kappa$) and the existence of non-trivial equilibria has been observed in other parimutuel wagering studies. For example, a non-trivial equilibrium exists in Watanabe's model only if the house take is sufficiently small (\cite{Takahiro199785}). Ottaviani and S\o renson make a similar observation (\cite{ottaviani2006timing}).

Intuitively, $\kappa >  0.5$ is necessary for the existence of an equilibrium in our framework because there are only two possible outcomes. Suppose that $\left( f^{\star} , a^{\star}   \right)$ is a pure-strategy Nash equilibrium. Regardless of her type, if a player believes that Outcome 1 will occur with probability $p$ and wagers on Outcome 1, it should be true that her final expected profit per unit bet on Outcome 1 is positive:
\begin{equation*}
\displaystyle\frac{ \kappa \,  p }{ P^{f^{\star}, a^{\star}}} - 1 > 0 . 
\end{equation*}
Similarly, if she wagers on Outcome 2, then
\begin{equation*}
 \displaystyle\frac{ \kappa \,   \left( 1 - p \right) }{  1 - P^{f^{\star} , a^{\star} } } - 1 > 0.
 \end{equation*}
 Positive amounts are wagered on both outcomes (see Section \ref{model sect}), implying that  
\begin{equation}\label{both profitable ineq}
\displaystyle\frac{ \kappa }{ P^{f^{\star}, a^{\star}}} - 1 > 0  \qquad \text{and} \qquad \displaystyle\frac{ \kappa   }{  1 - P^{f^{\star} , a^{\star} } } - 1 > 0 .
\end{equation}
Rearranging (\ref{both profitable ineq}) shows that $\kappa >  0.5$. More generally, it is easy to see that in a parimutuel betting game with $n$ outcomes and risk-neutral players who can elect not to bet, $\kappa > 1/n$ is necessary for the existence of an equilibrium.

Parimutuel wagering games in the literature occasionally possess multiple equilibria. For instance, Watanabe et al. adapt the work of Harsanyi and Selten to select one equilibrium out of several that arise in their atomic player model (\cite{wat+non+mori}). As mentioned in Section \ref{model sect}, Watanabe's continuum game model can feature multiple equilibria when $\mu \left( \left\{ p \right\} \right) > 0$  for some fixed $p$. We suspect that our equilibrium would no longer be unique, if we incorporated additional atomic players or relaxed our assumption that $\mu$ had a density, but we leave this issue to a future study.

We break our proof into \ref{step 5 orig} steps. Step \ref{step 1 orig} allows us to use Propositions \ref{diff equilib bets prop} and \ref{atom equilib bets prop}. It says that in an equilibrium, the total amount wagered by the diffuse players on each outcome is always positive. Step \ref{step 2 orig} formalizes our discussion above and shows that $\kappa >  0.5$ is necessary for the existence of an equilibrium.

To finish, we find an equivalent formulation of our original problem. More precisely, after presenting some preliminary notation in Steps \ref{step 3 alt 2} - \ref{step 4 alt 2}, we define our so-called {\it implied probability map} $\varphi$ in Step \ref{step 5 alt 2}. Our work in Steps \ref{step 6 alt varphi dec} - \ref{step 6 uniq fp} shows that this map has a unique fixed-point. Due to its construction, its fixed-point corresponds to a pure-strategy Nash equilibrium and vice versa. We can then conclude that our game has a unique equilibrium when $\kappa >  0.5$ in Steps \ref{step 4 orig} - \ref{step 5 orig}.

Our approach is motivated by the following observation: an equilibrium is essentially determined by the implied probability of Outcome 1. Given this quantity, we immediately recover the diffuse players' wagers from Proposition \ref{diff equilib bets prop}. Technically, we do not know how the diffuse players whose beliefs are indexed by $P^{f^{\star}, a^{\star}}/ \kappa$ or $\left( 1 - P^{f^{\star}, a^{\star}}\right) / \kappa$ behave, but this does not matter. By Proposition \ref{atom equilib bets prop}, we then identify the atomic player's wagers.

This observation leads to the definition of $\varphi$ in Step \ref{step 5 alt 2}. In a certain sense, our recipe is only meaningful at a fixed-point of $\varphi$, which underlies the correspondence just discussed. Our complete proof of Theorem \ref{main thm} can be found in Appendix \ref{main thm proof app}.

We close Section \ref{main result sect} with Corollaries \ref{atom pl exp prof pos cor}, \ref{reg equilib cor}, and \ref{impl prob kappa  0.5}. The first result says that the atomic player wagers on a particular outcome if and only if her final expected profit per unit bet on that outcome is positive. This is fairly obvious from \ref{NE def orig iii} of Definition \ref{maj min NE def}, and we basically assume this to be true during our discussion of Theorem \ref{main thm}.

Still, recall that Proposition \ref{atom equilib bets prop} identifies the outcome, if any, on which the atomic player wagers based upon only the diffuse players' bets. For instance, according to that result, the atomic player wagers on Outcome 1 in an equilibrium $\left( f^{\star} , a^{\star} \right)$ if and only if 
\begin{equation}\label{at bet o1 in equlib star}
q > \displaystyle\frac{ d_1^{\star} }{\kappa \left( d_1^{\star} + d_2^{\star} \right)} . 
\end{equation}
Since $a_1^{\star} > 0$, we might be concerned about the possibility that
\begin{equation*}
\frac{P^{f^{\star},a^{\star}}}{\kappa} \geq q > \displaystyle\frac{ d_1^{\star} }{\kappa \left( d_1^{\star} + d_2^{\star} \right)} .
\end{equation*}
The specific form of $a_1^{\star}$ in (\ref{atom bets out1 soln alone}) ultimately prevents this.

Notice that Proposition \ref{diff equilib bets prop} already {\it almost} implies the corresponding result for diffuse players. We say {\it almost} because of the undetermined behavior of the diffuse bettors whose final expected profit per unit bet on some outcome is 0. From that perspective, the atomic and diffuse players identify profitable wagering opportunities using the same criteria. Strategically, they just differ in how they size their equilibrium wagers.

\begin{cor}\label{atom pl exp prof pos cor}
Suppose that $\left( f^{\star} , a^{\star} \right)$ is an equilibrium. Then $a^{\star}_1 > 0$ if and only if 
\begin{equation*}
 \displaystyle\frac{ \kappa   q }{ P^{f^{\star}, a^{\star}}} - 1  > 0 , 
\end{equation*}
while $a^{\star}_2 > 0$ if and only if 
\begin{equation*}
\displaystyle\frac{ \kappa   \left( 1 - q \right) }{  1 - P^{f^{\star}, a^{\star}} }   -1 >0  . 
\end{equation*}
\end{cor}

\begin{proof}
We present the argument for the first case. The other is similar. 

Recall that $d_1^{\star}$, $d_2^{\star} >0$ (see Step \ref{step 1 orig} of Theorem \ref{main thm}'s proof). Assume that $a_1^{\star} > 0$. According to Proposition \ref{atom equilib bets prop}, (\ref{atom bets out1 ineq alone}) holds and $a_2^{\star} = 0$. Using the notation from Step \ref{step 4 alt 2} of Theorem \ref{main thm}'s proof, 
 \begin{equation*}
 P^{f^{\star}, a^{\star}} \leq  \displaystyle\frac{ \zeta_1 \left(  P^{f^{\star}, a^{\star}}\right)   + d_1^{\star} }{ \zeta_1 \left(  P^{f^{\star}, a^{\star}}\right)  + d_1^{\star}  + d_2^{\star}  } . 
  \end{equation*} 
That 
\begin{equation}\label{atom exp prof pub pos ineq}
 \displaystyle\frac{ \kappa   q }{ P^{f^{\star}, a^{\star}}} - 1  > 0 
\end{equation}
follows from Isaacs' work (\cite{isaacs}).

To prove the remaining direction, assume that (\ref{atom exp prof pub pos ineq}) is satisfied. Exactly one of (\ref{atom bets out1 ineq alone}), (\ref{atom bets out2 ineq alone}), and (\ref{atom bets noth ineq alone}) holds. We cannot have (\ref{atom bets out2 ineq alone}), since it would follow that $a_2^{\star} > 0$. Arguing as we just did, we would get 
\begin{equation*}
\displaystyle\frac{ \kappa   \left( 1 - q \right) }{  1 - P^{f^{\star}, a^{\star}} }   -1 >0  ,
\end{equation*}
which would lead to 
\begin{equation*}
\displaystyle\frac{ \kappa   q }{ P^{f^{\star}, a^{\star}}} - 1  < 0.
\end{equation*} 
(\ref{atom bets noth ineq alone}) cannot hold either, as this would give the contradiction
\begin{equation*}
q \leq \displaystyle\frac{ d_1^{\star} }{  \kappa \left( d_1^{\star} + d_2^{\star}  \right) }  = \frac{P^{f^{\star},a^{\star}} }{\kappa}. 
\end{equation*}
Hence, (\ref{atom bets out1 ineq alone}) holds and $a_1^{\star} > 0$.

\end{proof}

To explain our next result, suppose that some player is wagering on a particular outcome. If another player believes that this outcome will occur with higher probability than the original player, Corollary \ref{reg equilib cor} says that the new player also wagers on the outcome. Recall from Section \ref{lit review sect} that Watanabe called an equilibrium with this property {\it regular} (\cite{Takahiro199785}). All equilibria in Watanabe's framework and Ottaviani and S\o renson's framework are regular (\cite{ottaviani2006timing}; \cite{Takahiro199785}). 

One might suspect that such a result generally holds, but this is not the case (\cite{wat+non+mori}). To the best of our knowledge, regularity has only been consistently observed in the literature when each player's initial wealth is negligible (\cite{ottaviani2006timing}; \cite{Takahiro199785}). Corollary \ref{reg equilib cor} shows that our model is an example of a parimutuel wagering game in which the equilibrium is regular, even when an atomic player is active. It is an immediate consequence of our observation that both atomic and diffuse players decide to wager on some outcome by determining whether or not the final expected profit per unit bet on that outcome is positive.

\begin{cor}\label{reg equilib cor}
An equilibrium $\left( f^{\star} , a^{\star} \right)$ is always regular. 
\end{cor}

\begin{proof}
Since $d_1^{\star}$, $d_2^{\star} >0$ by Step \ref{step 1 orig} of Theorem \ref{main thm}'s proof, the result directly follows from Proposition \ref{diff equilib bets prop} and Corollary \ref{atom pl exp prof pos cor}.

\end{proof}

Section \ref{main result sect}'s last result says that the implied probability of Outcome 1 tends to $ 0.5$ as the house take approaches $ 0.5$, regardless of the other parameters that we choose for our model. In fact, the convergence is uniform.

Initially, this finding may appear rather odd. For example, it is easy to ensure that $P^{f^{\star},a^{\star}}$ lies between $49.9\%$ and $50.1\%$ when $q= 0$ and the $\mu$-mass of $\left[ 0 , 1 \right]$ is {\it almost} entirely concentrated near $p = 0$. If essentially the whole population believes that Outcome 2 is guaranteed to occur, how can the total amounts wagered on each outcome be roughly equal?

Our discussion of Theorem \ref{main thm} outlines the key intuition. Simply notice that instead of rearranging (\ref{both profitable ineq}) to show that $\kappa >  0.5$, we can show that $P^{f^{\star}, a^{\star}} \in \left( 1 - \kappa , \kappa \right)$. This holds even in the extreme scenario where virtually all of the initial wealth is held by those who believe that Outcome 2 is a sure bet. Still, our first instinct has some merit: Here, $P^{f^{\star}, a^{\star} } \approx 1 - \kappa$ for all $\kappa$ (see Appendix \ref{corr app}).

\begin{cor}\label{impl prob kappa  0.5}
Fix $\mu$, $q$, and $w$ and consider the map defined on $\left(  0.5 , 1 \right)$ by
\begin{equation*}
\kappa \mapsto P^{f^{\star}, a^{\star}} .
\end{equation*}
As $\kappa \downarrow  0.5$, the values of the map approach $ 0.5$.

\end{cor}

\begin{proof}
Simply notice that the map is well-defined by Theorem \ref{main thm} and that $P^{f^{\star}, a^{\star}}  \in \left( 1 - \kappa , \kappa \right)$ by Step \ref{step 5 orig} of its proof.

\end{proof}

\section{Numerical Results}\label{leb meas sect}



The theoretical results from Section \ref{main result sect} allow us to return to our central question: How do large-scale participants in parimutuel wagering events affect the house and ordinary bettors? We explore this issue by analyzing several concrete scenarios (see Examples \ref{exmp 1} - \ref{exmp 4}).

We use the house's revenue to quantify the atomic player's impact on the house. Given an equilibrium $\left( f^{\star} , a^{\star} \right)$, the house's revenue is simply the product of the house take and the total amount wagered:
\begin{equation}\label{house rev eqn}
\left( 1 - \kappa \right) \left( d_1^{\star} + d_2^{\star} + a_1^{\star} + a_2^{\star} \right) . 
\end{equation}
Notice that this quantity is deterministic and does not depend on the {\it actual} probability of Outcome 1, as the house collects (\ref{house rev eqn}) regardless of which outcome occurs.

To quantify the atomic player's effect on diffuse bettors, we use one of two quantities. In Examples \ref{exmp 1} - \ref{exmp 2}, we select values for the {\it actual} probability of Outcome 1. Making this choice lets us compute the {\it actual} total expected profit of the diffuse players. If $\left( f^{\star} , a^{\star} \right)$ is an equilibrium and the {\it actual} probability of Outcome 1 is $\bar{p}$, then it is given by
\begin{equation}\label{act exp prof db eqn}
d_1^{\star}  \bigg( \displaystyle\frac{ \kappa  \,  \bar{p} }{ P^{f^{\star}, a^{\star}}} - 1 \bigg) + d_2^{\star} \left( \displaystyle\frac{ \kappa \,   \left( 1 -  \bar{p} \right) }{  1 - P^{f^{\star}, a^{\star}} } - 1 \right) .
\end{equation}

While we use (\ref{act exp prof db eqn}) to describe how the atomic player affects the diffuse players in Examples \ref{exmp 1} - \ref{exmp 2}, we cannot do so in Examples \ref{exmp 3} - \ref{exmp 4}. The reason is that we make no assumption about the {\it actual} probability of Outcome 1 in the latter situations. Instead, we quantify the impact on diffuse bettors using their total {\it subjective} expected profit, which is given by 
\begin{equation}\label{subj exp prof db eqn}
\displaystyle\int_{0}^{1} \left[ f_1^{\star} \left( p \right)  \bigg( \displaystyle\frac{ \kappa \,   p }{ P^{f^{\star}, a^{\star}}} - 1 \bigg) + f_2^{\star} \left( p \right)  \left( \displaystyle\frac{ \kappa \, \left( 1 - p \right) }{  1 - P^{f^{\star}, a^{\star}} } - 1 \right) \right] d \mu \left( p \right)
\end{equation}
in an equilibrium $\left( f^{\star} , a^{\star} \right)$. Here, we merely compute each diffuse player's expected profit according to her beliefs and aggregate the results over all diffuse players.

From Section \ref{intro sect}, recall that the standard narrative says that the presence of the atomic player should increase the house's revenue but decrease the diffuse players' total expected profit ({\it actual} or {\it subjective}, as applicable). The eventual decline in the house's revenue should only be seen over time, not in our static game model.

Technically, since we specified that $w > 0$ in Section \ref{model sect}, the atomic player can never be absent in our framework. Still, we can model the atomic player's absence by choosing an extremely low value for $w$, say $w = 10^{-10}$. Her wagers are then too small to materially affect any equilibrium quantities. We do this for all of the Case 1's in Examiples \ref{exmp 1} - \ref{exmp 4}. The atomic player is present, that is, $w > > 0$, in all of our upcoming Case 2's.

The following collection of measures is convenient for our purposes.
\begin{defn}\label{dens fxn exmp defn}
For $n \geq 1$, let $\mu_n$ be the Borel measure on $\left[ 0 , 1 \right]$ whose density $g_n$ is defined by 
\begin{align*}
g_n \left( p \right) &= \left\{ \begin{array}{cc}
-2 n \left( n - 1 \right) p  + 2 \left( n -1 \right) + 1/n & \quad \text{if } \, \,  p <  1/n \\
1/n & \quad \text{if } \, \, p \geq 1/n   \\
\end{array}   \right. .
\end{align*}
\end{defn}
In Figure \ref{fig: gn varying n plot}, we give the plots of $g_n$ for $n = 1$, 3, and 9. Here are the key observations:
\begin{enumerate}[label=(\roman*)]
\item \label{gn cont pos} $g_n$ is continuous and positive on $\left[ 0 , 1 \right]$.
\item \label{un meas 1} $\mu_n \left( \left[ 0 , 1 \right] \right) = 1$ for all $n$.
\item \label{mu1 is Leb} $\mu_1$ is the Lebesgue measure on $\left[ 0 , 1 \right]$.
\item \label{gn conv Dirac} $g_n$ converges in distribution to the Dirac delta function as $n \uparrow \infty$. 
\end{enumerate}
\ref{gn cont pos} and \ref{un meas 1} ensure that $\mu_n$ is a suitable candidate for the measure describing the initial wealth of the diffuse bettors, that is, all of our results apply when $\mu = \mu_n$. In this case, \ref{mu1 is Leb} says that the initial wealth of the diffuse bettors is uniformly distributed when $n =1$. \ref{gn conv Dirac} says that their initial wealth is increasingly concentrated among those who believe that Outcome 2 will occur with high probability as $n$ increases. Another key feature is that the diffuse bettors' total initial wealth is always equal to 1 (see \ref{un meas 1}). 

\begin{figure}[ht]
\centering
\begin{minipage}[b]{0.45\linewidth}
\includegraphics[scale=0.35]{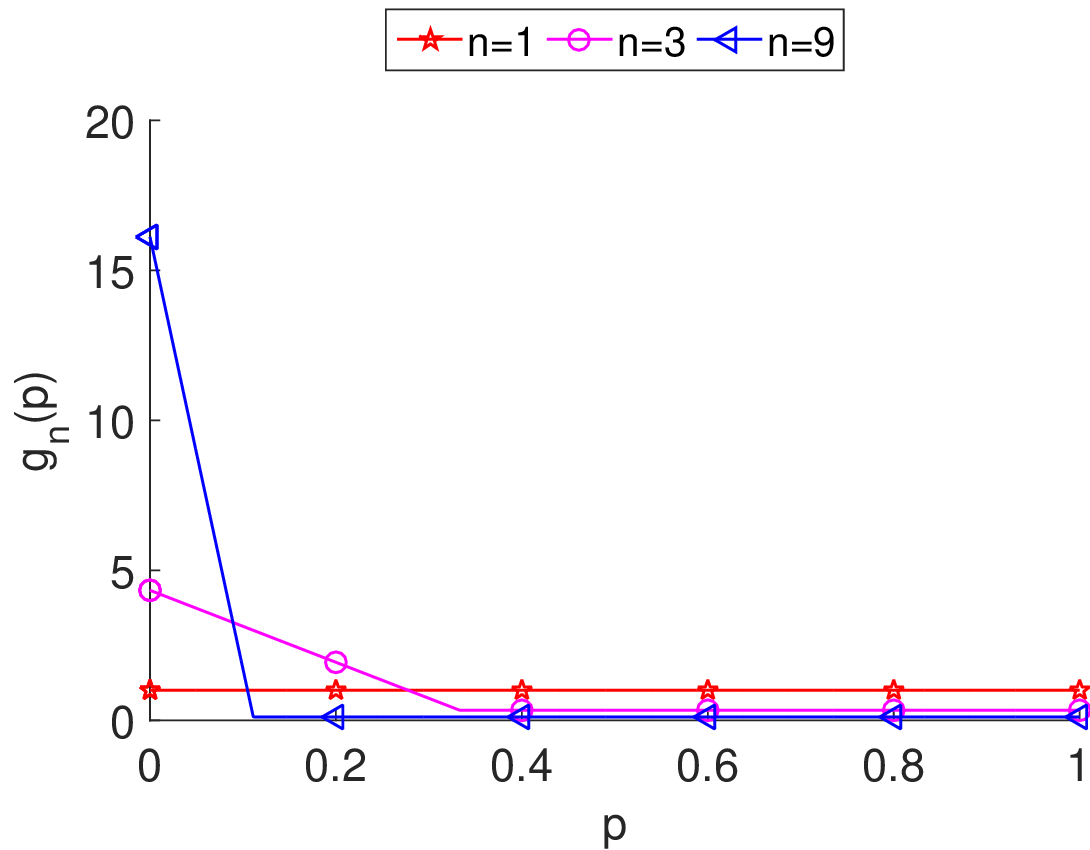}
\caption{}
\label{fig: gn varying n plot}
\end{minipage}
\end{figure}

Before we proceed, we remark that all of our figures are generated with the help of the ideas in Theorem \ref{main thm}'s proof. More precisely, assume that $\kappa \in \left(  0.5 , 1 \right)$. The function $\varphi$ defined in Step \ref{step 5 alt 2} has a unique fixed-point. Since $\varphi$ is also continuous and decreasing, we can efficiently approximate this value with arbitrary precision using binary search. Step \ref{step 4 orig} shows how to reconstruct the pure-strategy Nash equilibrium $\left( f^{\star} , a^{\star} \right)$, given such an estimate.

\begin{example}\label{exmp 1}

We now show that both effects predicted by the usual narrative can be observed.

  In Cases 1 and 2, we set $\mu = \mu_1$ and $q = 0.9$. We assume that the atomic player's beliefs are exactly correct, i.e., the {\it actual} probability that Outcome 1 will occur is also 0.9. The only difference between Cases 1 and 2 is that $w = 10^{-10}$ in the former but $w = 1$ in the latter.

Recall that choosing $\mu = \mu_1$ means that the diffuse bettors' initial wealth is uniformly distributed. Since $q = 0.9$, the atomic player (correctly) believes that Outcome 1 is quite likely.

In Figure \ref{fig: Diff Pl EP Exmp 1}, we plot the diffuse players' {\it actual} total expected profit. Collectively, the diffuse players beliefs are rather inaccurate, so it is not surprising that their expect profit is negative. Still, as $\kappa \uparrow 1$, the diffuse players become increasingly worse off in Case 2. Intuitively, the atomic player quickly raises her bet on Outcome 1 as $\kappa \uparrow 1$. The implied probability of Outcome 1 rises,\footnote{We remark that $P^{f^{\star}, a^{\star}} \in \left[ 0.5, 0.7 \right]$ for all $\kappa \in \left( 0.5 , 1 \right)$ in Example \ref{exmp 1}. In particular, the {\it favorite-longshot bias} results, regardless of the house take.} causing more diffuse players to bet on Outcome 2 and less to bet on Outcome 1. Since the {\it actual} probability of Outcome 1 is 0.9, this transition negatively affects the diffuse players.

We plot the house's revenue in Figure \ref{fig: House Rev Exmp 1}. The house's revenue is higher in Case 2 than in Case 1 for all $\kappa \in \left( 0.5, 1 \right)$, as a result of the higher wagering totals in Case 2.


\begin{figure}[ht]
\centering
\begin{minipage}[b]{0.45\linewidth}
\includegraphics[scale=0.35]{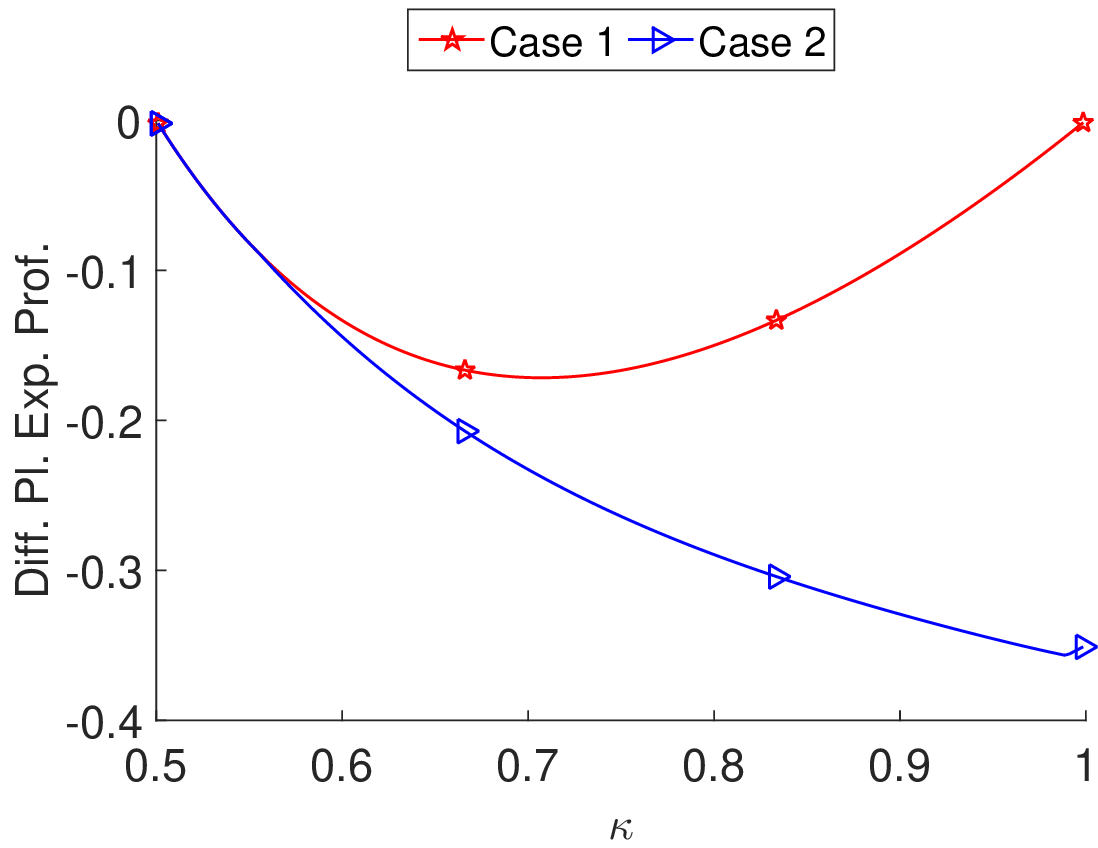}
\caption{}
\label{fig: Diff Pl EP Exmp 1}
\end{minipage}
\qquad \quad
\begin{minipage}[b]{0.45\linewidth}
\includegraphics[scale=0.35]{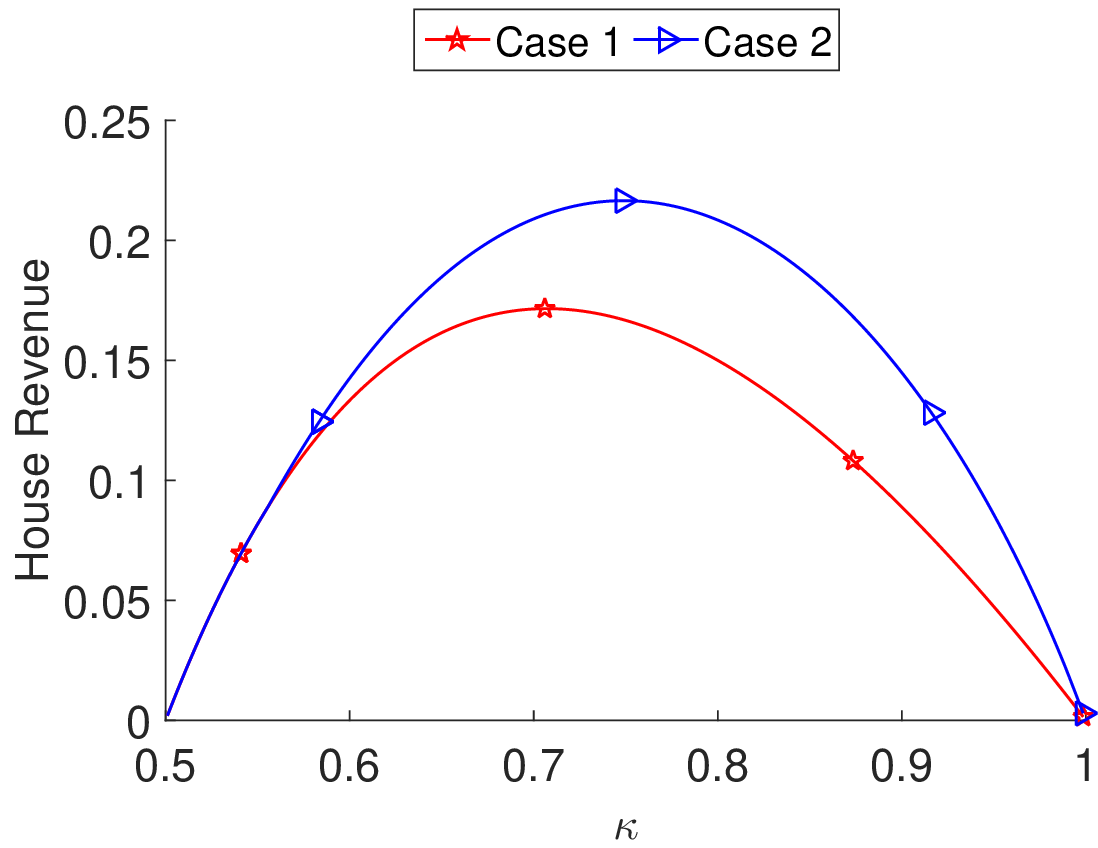}
\caption{}
\label{fig: House Rev Exmp 1}
\end{minipage}
\end{figure}

\end{example}

\begin{example}\label{exmp 2}

Still, there are cases in which diffuse players are positively affected by the activities of the atomic player.

We now choose $\mu = \mu_1$ and $q = 0.57$ for Cases 1 and 2. The {\it actual} probability of Outcome 1 is 0.47. The only difference between the two scenarios is that $w = 10^{-10}$ in Case 1, while $w = 1$ in Case 2. 

Compared to Example \ref{exmp 1}, Outcome 1 is slightly less likely here. Also, we still assume that the atomic player's forecast is quite accurate, but her prediction is no longer perfect.

We plot the diffuse players' {\it actual} expected profit in Figure \ref{fig: Diff Pl EP Exmp 2}. The graphs for Cases 1 and 2 are the same for most values of $\kappa$; however, we see that the diffuse players' expected profit is {\it higher} in Case 2 when $\kappa$ is large enough. Roughly, the atomic player does not start betting (on Outcome 1) until the house take is low, since she is nearly ambivalent. This explains why the curves are initially identical. When the atomic player begins to wager on Outcome 1, the diffuse players reshuffle their bets as in Example \ref{exmp 1}.\footnote{Like Example \ref{exmp 1}, we see the {\it favorite-longshot bias} here because $P^{f^{\star}, a^{\star}} \in \left[ 0.5, 0.53 \right]$ for all $\kappa \in \left( 0.5 , 1 \right)$.} The shift benefits them, essentially because the atomic player wagers on the {\it wrong} outcome.

In Figure \ref{fig: House Rev Exmp 2}, we plot the house's revenue. We see an extremely small improvement in Case 2 when $\kappa$ is large, but the graphs are almost indistinguishable, visually. As in Example \ref{exmp 1}, the increase corresponds to the increase in the total amount wagered. It is slight, as the atomic player's uncertainty about what will occur causes her to bet very little in Case 2, even for $\kappa \approx 1$.


\begin{figure}[ht]
\centering
\begin{minipage}[b]{0.45\linewidth}
\includegraphics[scale=0.35]{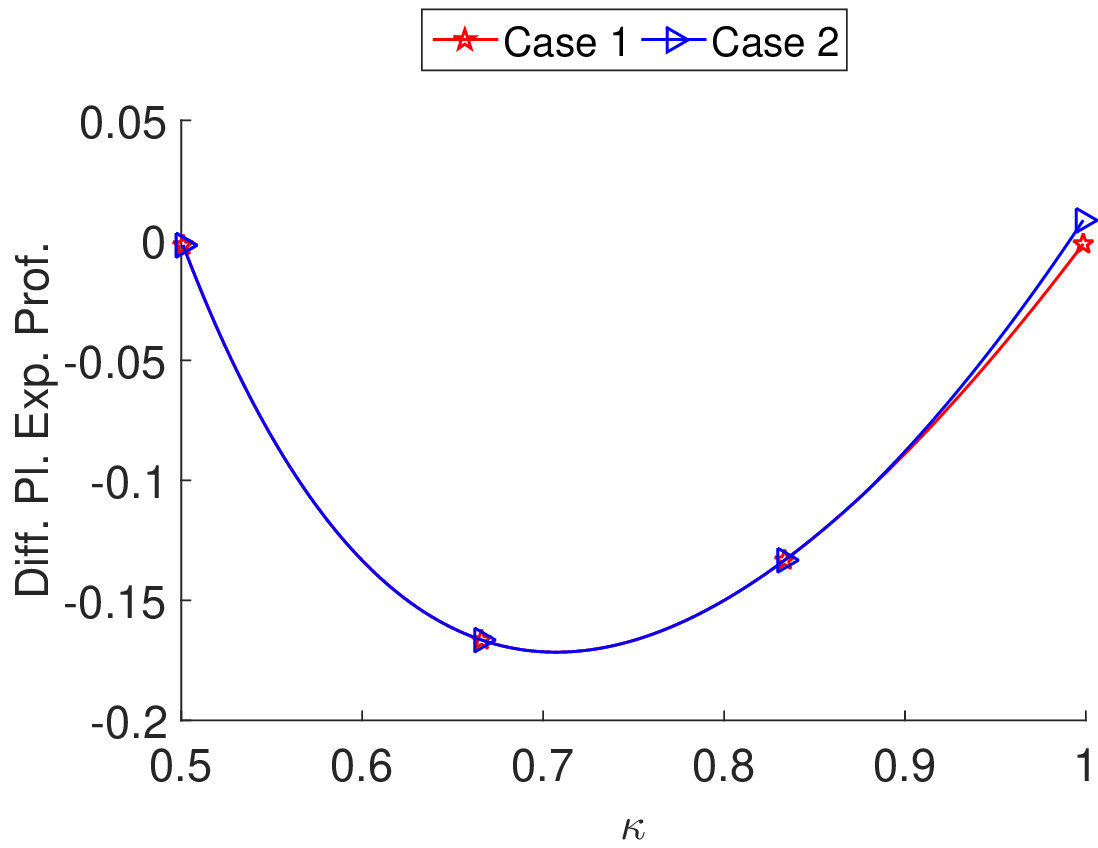}
\caption{}
\label{fig: Diff Pl EP Exmp 2}
\end{minipage}
\qquad \quad
\begin{minipage}[b]{0.45\linewidth}
\includegraphics[scale=0.35]{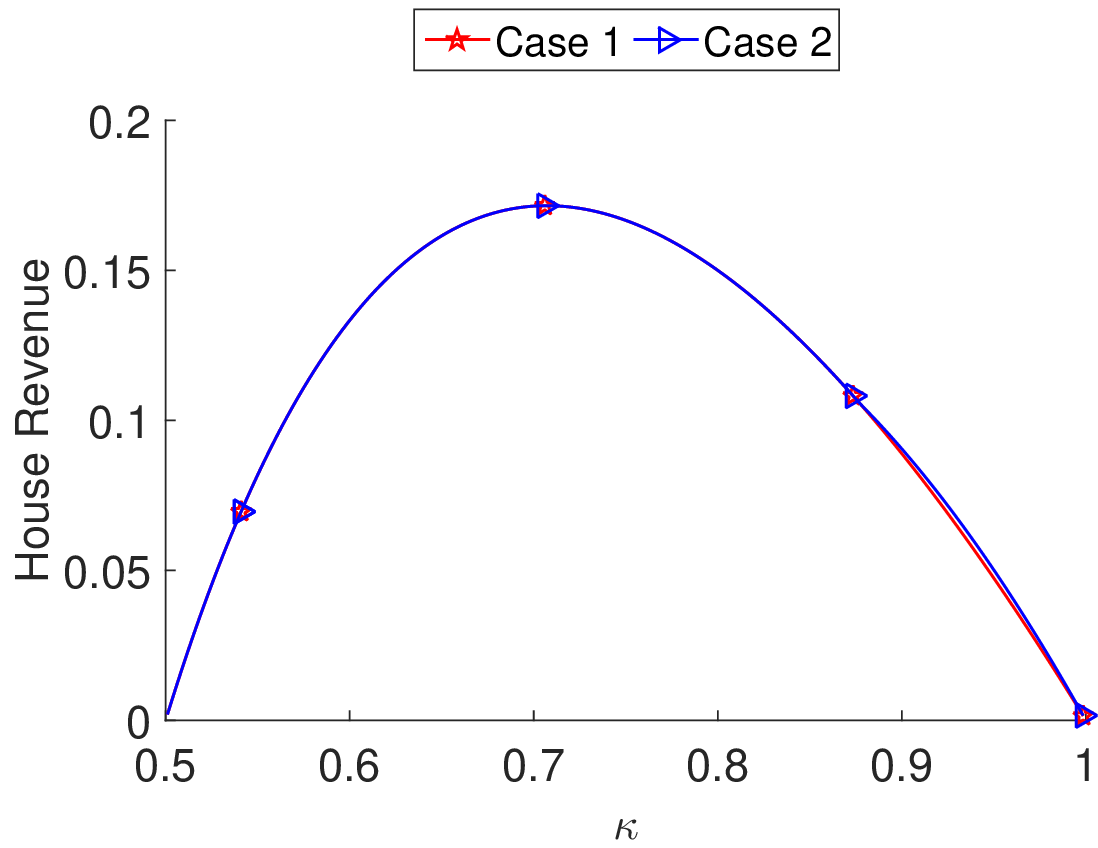}
\caption{}
\label{fig: House Rev Exmp 2}
\end{minipage}
\end{figure}

\end{example}

\begin{example}\label{exmp 3}

We can argue that the diffuse players are better off in the presence of the atomic player, even without making assumptions about the {\it actual} probability of Outcome 1.

In Cases 1 and 2, we now choose $\mu = \mu_{10}$ and $q = 0.95$. As in Examples \ref{exmp 1} - \ref{exmp 2}, we set $w = 10^{-10}$ in Case 1 and $w = 1$ in Case 2. The atomic player believes that Outcome 1 is highly likely. Collectively, the diffuse players believe that Outcome 2 will probably happen: over $90\%$ of their initial wealth is held by those who believe that the probability of Outcome 2 is at least 0.9. We make no judgment about the accuracy of the players' beliefs.

We plot the diffuse players' {\it subjective} expected profit in Figure \ref{fig: Diff Pl EP Exmp 3}. The graphs are the same for low $\kappa$, but eventually, the diffuse players' {\it subjective} expected profit is much higher in Case 2. Intuitively, the atomic player raises her wager on Outcome 1 as $\kappa \uparrow 1$, since she believes that Outcome 1 will occur. The implied probability of Outcome 1 then rises, a boon to those who believe that Outcome 2 will occur. This includes {\it most} of the diffuse players.

In Figure \ref{fig: House Rev Exmp 3}, we plot the house's revenue. The house's revenue in Case 2 is at least as large as its revenue in Case 1 for all values of $\kappa$. Often, the improvement is significant. Roughly, there is too much agreement among the diffuse bettors. They wager more in the presence of the atomic player, {\it believing} that they can profit from her {\it supposed} wagering mistakes. The pool size increases, leading to greater revenue for the house in Case 2.


\begin{figure}[ht]
\centering
\begin{minipage}[b]{0.45\linewidth}
\includegraphics[scale=0.35]{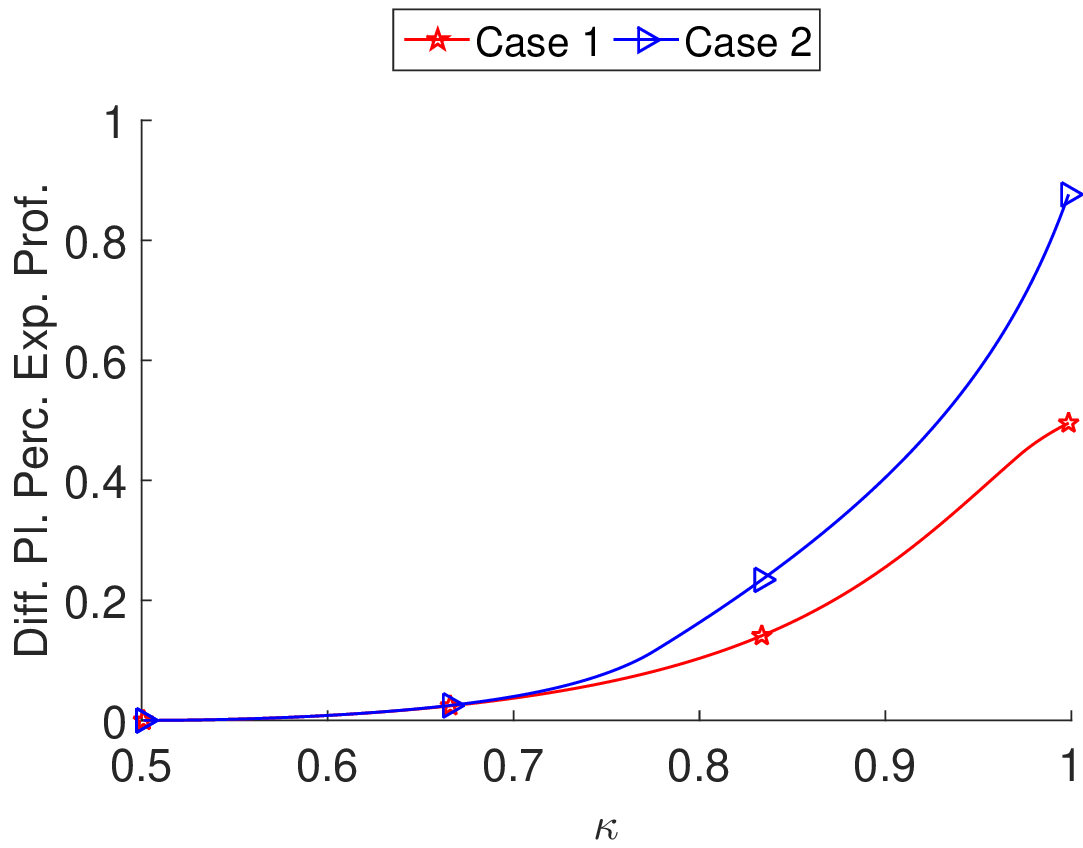}
\caption{}
\label{fig: Diff Pl EP Exmp 3}
\end{minipage}
\qquad \quad
\begin{minipage}[b]{0.45\linewidth}
\includegraphics[scale=0.35]{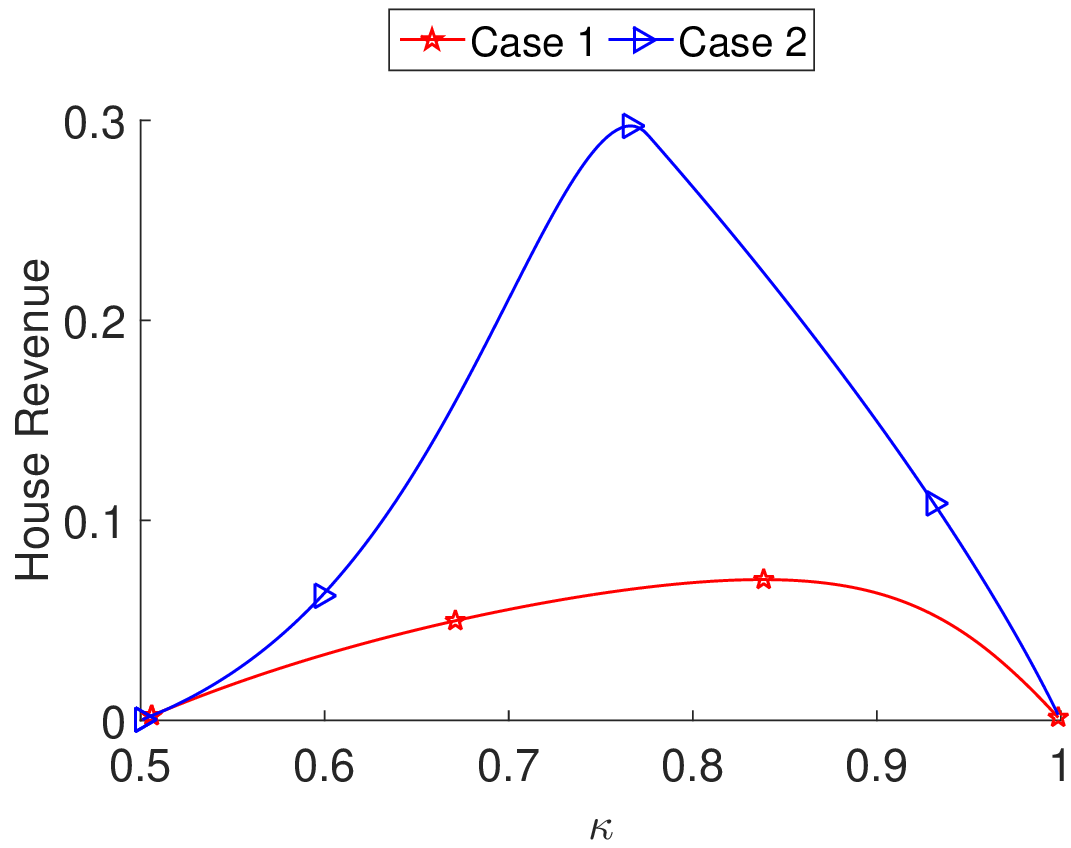}
\caption{}
\label{fig: House Rev Exmp 3}
\end{minipage}
\end{figure}

\end{example}

\begin{example}\label{exmp 4}

In keeping with the standard narrative, it seems like the house is always immediately better off because of the atomic player. We can cast some doubt on this too.

We used the same $\mu$ for Cases 1 and 2 in Examples \ref{exmp 1} - \ref{exmp 3}. This choice captures the idea that the presence of the atomic player should not affect the diffuse players' initial wealth in a static parimutuel wagering game.\footnote{Over the course of many events, diffuse players may stop participating because of the atomic player, making this intuition questionable in another context (see Section \ref{intro sect}).}

Alternatively, it might be reasonable to fix the distribution of initial wealth across the {\it entire} population, not just the diffuse population. For instance, we could specify that the amount held by those who believe that Outcome 1 will occur is roughly the same as the amount held by those who believe that Outcome 2 will occur. We could then compare the situations in which the wealth is held by only diffuse players and in which some wealth is held by the atomic player. This is the approach we now take.

In Case 1, $w = 10^{-10}$ and the density of $\mu$ is defined by  
\begin{equation*}
p \mapsto \frac{ g_{100} \left(p \right) + g_{100} \left(1 -p \right)  }{2} .
\end{equation*}
For Case 2, we set $w  = 1$ and $\mu = \mu_{100}$. We choose $q = 1$ for both scenarios and, again, make no assumption about Outcome 1's {\it actual} probability.

Intuitively, {\it half} of the diffuse players in Case 1 believe that Outcome 1 is going to occur, while the other {\it half} believes that Outcome 2 will occur. For Case 2, the diffuse players {\it all} believe that Outcome 2 is going to occur. The atomic player, whose wealth is equal to the collective wealth of the diffuse players, believes that Outcome 1 will occur.

Notice that the diffuse players' total initial wealth is equal in Cases 1 and 2; however, the wealth of the entire population in Case 2 is twice what it is in Case 1. This is consistent with our parameter selections in Examples \ref{exmp 1} - \ref{exmp 3}. When studying the house's revenue here, we might have also chosen $\mu = 0.5 \times \mu_{100} $ and $w  = 0.5$ in Case 2, ensuring that the wealth of the entire population is identical in both scenarios. We comment on this shortly.

One might suspect that Cases 1 and 2 are quite similar, but this is not true. In Figure \ref{fig: Diff Pl EP Exmp 4}, we plot the diffuse players' {\it subjective} expected profit. For $\kappa \approx 1$, their {\it subjective} expected profit is higher in Case 2 than in Case 1. Otherwise, it is lower for all $\kappa$, and often, the drop is significant.

To explain this observation, we plot the implied probability of Outcome 1 in Figure \ref{fig: Impl Prob Exmp 4}. The implied probability of Outcome 1 is about 0.5 for all $\kappa$ in Case 1. For Case 2, it is almost $1 - \kappa$ for low $\kappa$ but approaches $0.5$ as $\kappa \uparrow 1$.

Since the diffuse players believe that Outcome 2 will occur but the atomic player believes that Outcome 1 will occur, the intuition appears to be as follows. Roughly, the diffuse players are willing to tolerate unfavorable betting conditions more than the atomic player. Despite the fact that each type of player is {\it basically} sure that the outcome they are betting on will occur, the diffuse players in Case 2 raise the size of their wagers much faster than the atomic player when $\kappa$ is low. The diffuse players do this at the expense of their own {\it subjective} expected profit, which is nearly zero in Case 2 until $\kappa \approx 0.84$. The atomic player does not substantially raise her wager on Outcome 1 until the values of the implied probability of Outcome 1 and $\kappa$ make her {\it subjective} expected profit per unit bet very high.

One could argue that the atomic player's strategy is superior to the strategies employed by the diffuse players, although we have made no assumption about the accuracy of her prediction. Perhaps the reason is that unlike the diffuse players, she considers her individual impact on the other wagerers because of her substantial wealth.

In Figure \ref{fig: House Rev Exmp 4}, we plot the house's revenue. We see that the house's revenue is higher in Case 1 for small $\kappa$ but higher in Case 2 for large $\kappa$. After selecting $\mu = 0.5 \times \mu_{100} $ and $w  = 0.5$ in Case 2 (see our explanation above), we re-plot the house's revenue in Figure \ref{fig: Norm House Rev Exmp 4}. Now the house's revenue is higher in Case 1 for all $\kappa$. Regardless of our normalization, the key point is that the house's maximum revenue is always much higher in Case 1.

Intuitively, the diffuse players appear to have a greater tolerance for poor betting conditions than the atomic player, as discussed previously. They are willing to bet even when $\kappa \approx 0.5$, and consequently, about $99\%$ of the diffuse players are wagering in Case 1 when $\kappa \approx 0.51$. The house loses revenue by raising $\kappa$, since the entire population has already wagered {\it almost} everything that it can. In Case 2, the atomic player's reluctance to significantly raise her wager on Outcome 1 until betting conditions improve means that the total amount wagered is low for most $\kappa$. The pool is large only if $\kappa$ is quite high, so the house does not collect much.

In Examples \ref{exmp 1} - \ref{exmp 3}, the house's revenue in Case 2 is at least as great as the house's revenue in Case 1 for all $\kappa$. Of course, the house's maximum revenue is then higher in Case 2 for these studies. Here, the pointwise analysis is not as clean, leading us to more directly consider the house's strategic behavior.

Thus far, we have understood the house take to be exogenously determined. One way to relax this assumption is to break our game into two stages. In the first stage, the betting organizer chooses a value for the house take in order to maximize her revenue. The second stage is identical to our current setup.

Under the new framework, the house needs to account for more than the distribution of initial wealth across the entire population. The house must also consider how the initial wealth is distributed across the population {\it for each type of player}. Independent of our normalization of the entire population's wealth, the house's revenue is maximized at $\kappa \approx 0.506$ in Case 1 and at $\kappa \approx 0.839$ in Case 2. Hence, the house selects these values of $\kappa$ in Cases 1 and 2, respectively.

In Figure \ref{fig: Diff Pl EP Exmp 4}, the diffuse players' {\it subjective} expected profit is about 0.0085 when $\kappa \approx 0.506$ in Case 2, while it is about 0.023 when $\kappa \approx 0.839$ in Case 2. These numbers arise from our original parameter choices, ensuring that the total initial wealth of the diffuse players is 1 in both scenarios. The diffuse players are not well off in either Case 1 or Case 2. We could still argue that the setup of Case 2 is to their advantage, in contrast to our earlier pointwise analysis. Roughly, because the atomic player is less willing to bet under poor wagering conditions, the house more easily preys upon the diffuse players when she is absent.

\begin{figure}[ht]
\centering
\begin{minipage}[b]{0.45\linewidth}
\includegraphics[scale=0.35]{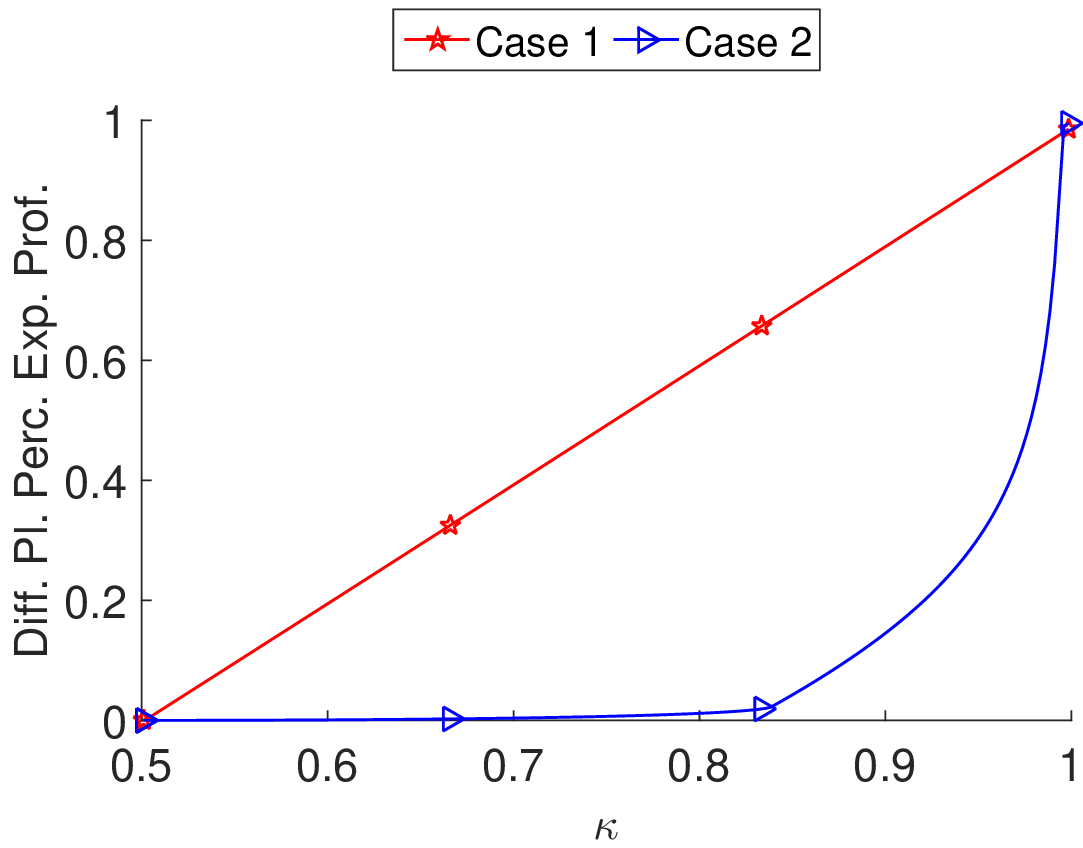}
\caption{}
\label{fig: Diff Pl EP Exmp 4}
\end{minipage}
\qquad \quad
\begin{minipage}[b]{0.45\linewidth}
\includegraphics[scale=0.35]{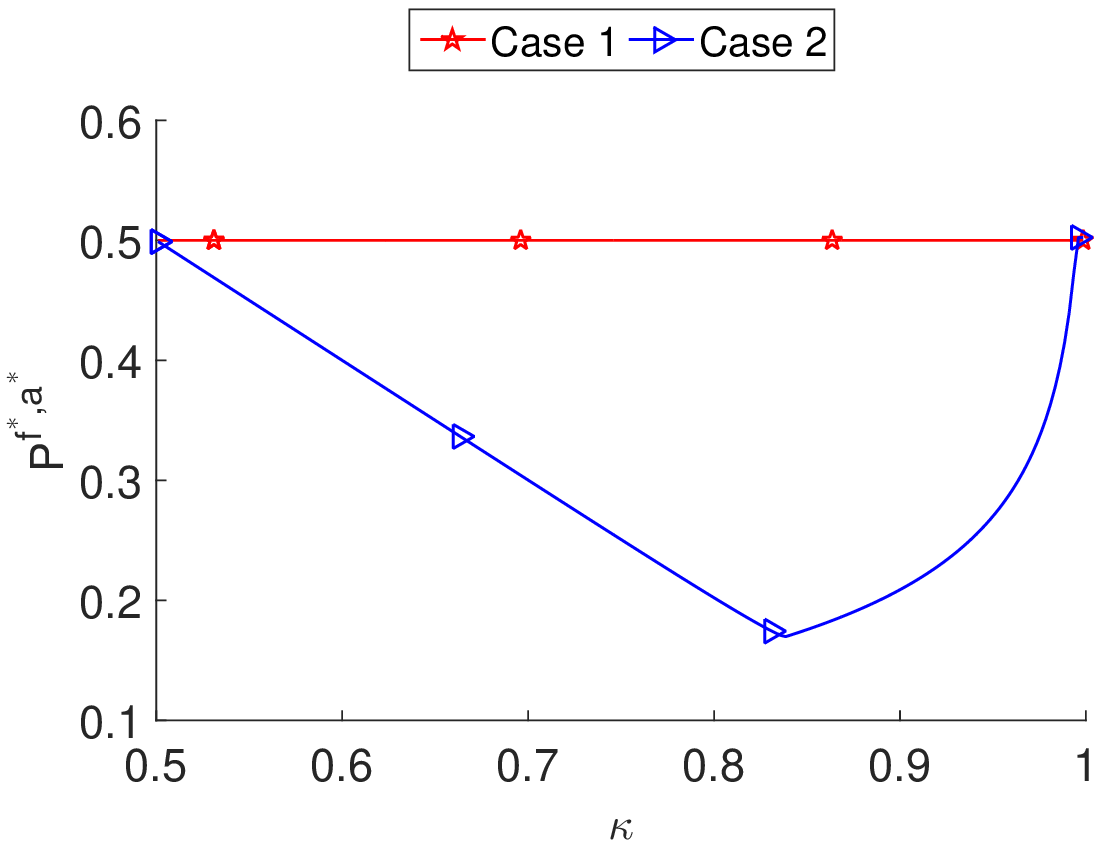}
\caption{}
\label{fig: Impl Prob Exmp 4}
\end{minipage}
\end{figure}

\begin{figure}[ht]
\centering
\begin{minipage}[b]{0.45\linewidth}
\includegraphics[scale=0.35]{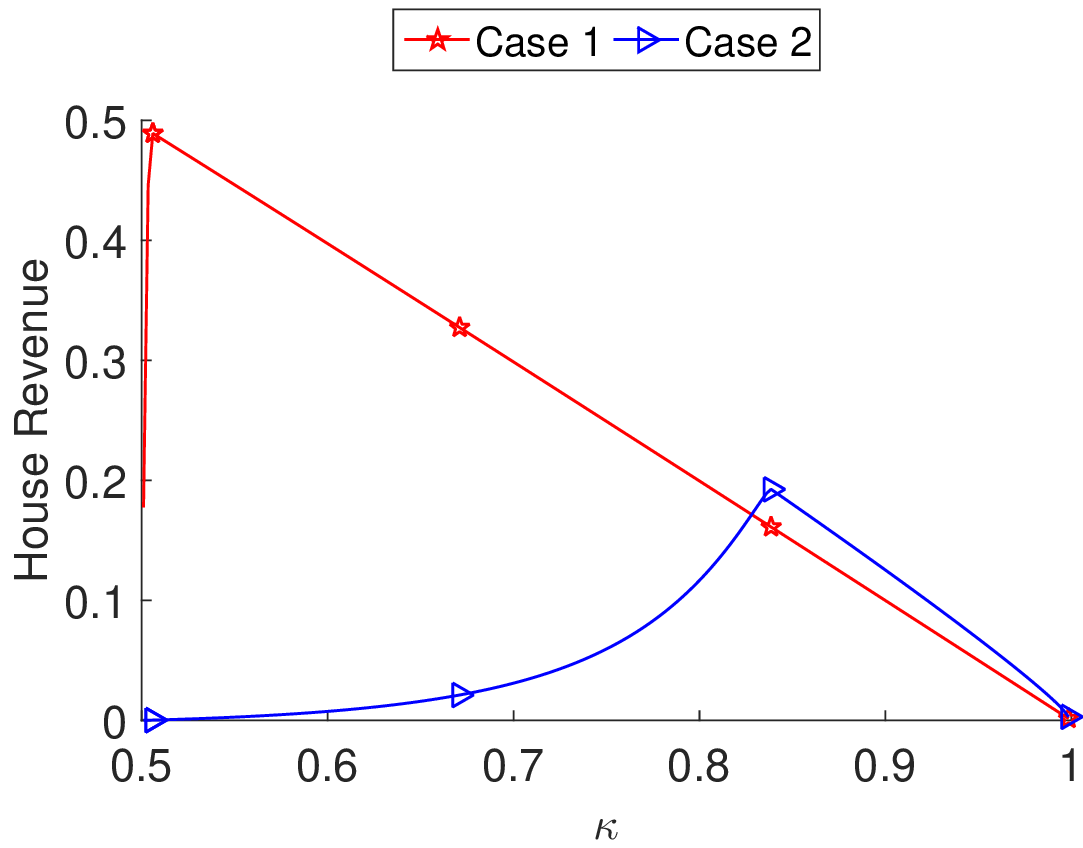}
\caption{}
\label{fig: House Rev Exmp 4}
\end{minipage}
\qquad \quad
\begin{minipage}[b]{0.45\linewidth}
\includegraphics[scale=0.35]{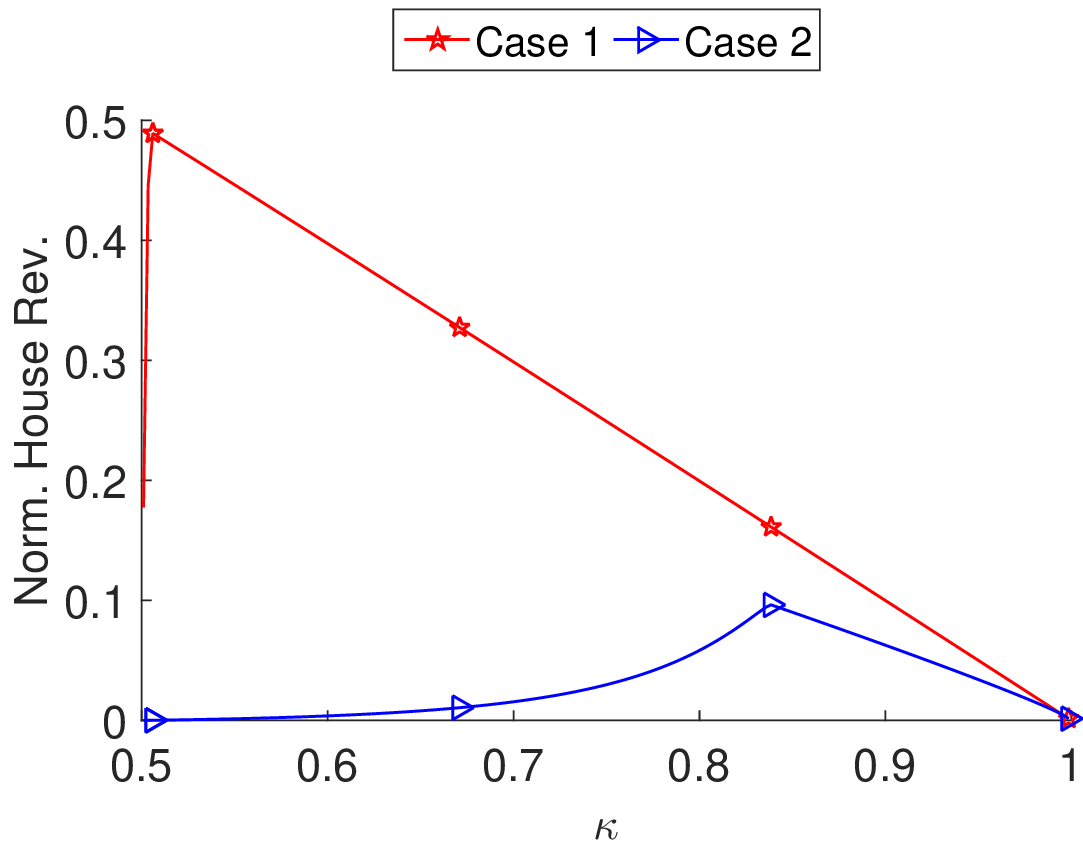}
\caption{}
\label{fig: Norm House Rev Exmp 4}
\end{minipage}
\end{figure}

\end{example}

\appendix

\section{Properties of $P^{f^{\star}, a^{\star}}$}\label{corr app}

We briefly revisit our discussion of Corollary \ref{impl prob kappa  0.5} by studying the properties of $P^{f^{\star}, a^{\star}}$ as a function of $\kappa$. We hope this highlights a few theoretical features of our model, though we do not relate our findings in Appendix \ref{corr app} to our central question. To generate these figures, we use the approach and notation of Section \ref{leb meas sect}.

Recall our claim that in an extreme case where nearly all of the initial wealth is held by those who believe that Outcome 2 is highly likely, $P^{f^{\star}, a^{\star} } \approx 1-  \kappa$ for all $\kappa$. We illustrate this in Figure \ref{fig: P1 star kappa All O2}, which depicts the map $\kappa \mapsto P^{f^{\star}, a^{\star} }$ when $\mu = \mu_{100}$, $q = 0$, and $w = 1$. Here, over 99\% of the diffuse players' initial wealth belongs to those who believe that the probability of Outcome 2 is at least 0.99. The atomic player believes that Outcome 2 is guaranteed to occur. As expected, the plot is visually indistinguishable from a plot of the map $\kappa \mapsto 1- \kappa$.

Often, it is more difficult to broadly describe attributes of the function $\kappa \mapsto P^{f^{\star}, a^{\star}}$. We know that $P^{f^{\star}, a^{\star}} \in \left( 1 - \kappa, \kappa \right)$, but our players' heterogeneity allows for a wide range of possibilities within these bounds. There is a rich interplay between their differing beliefs, effects on (\ref{payouts per unit bet words}), and wealth constraints. By plotting the map $\kappa \mapsto P^{f^{\star}, a^{\star}}$ under varying assumptions on $q$, $w$, and $\mu$, Figure \ref{fig: Wild P1star graph} displays a few of the myriad possibilities. For instance, the density $g$ of the $\mu$ used to generate Line $A$ is a linear combination of Gaussian densities with different means. Line $A$'s oscillations arise because of $g$'s distinct peaks.

\begin{figure}[ht]
\centering
\begin{minipage}[b]{0.45\linewidth}
\includegraphics[scale=0.35]{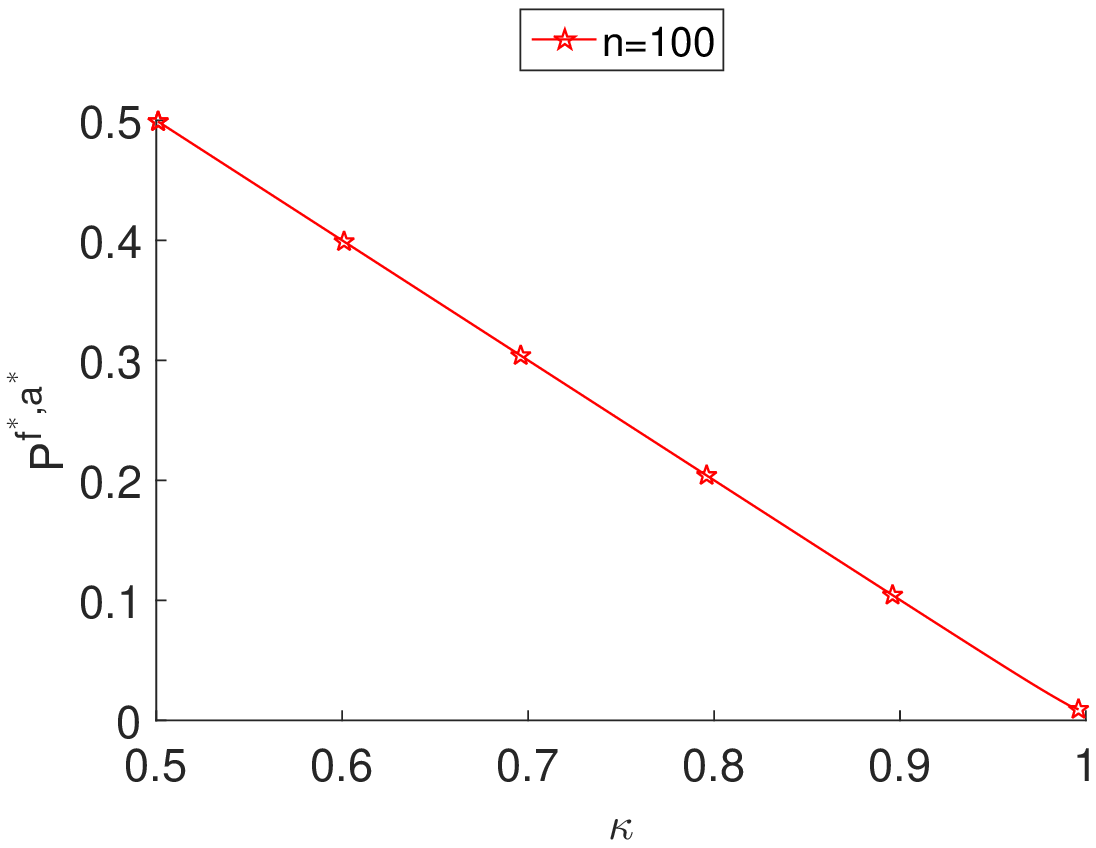}
\caption{}
\label{fig: P1 star kappa All O2}
\end{minipage}
\qquad \quad
\begin{minipage}[b]{0.45\linewidth}
\includegraphics[scale=0.35]{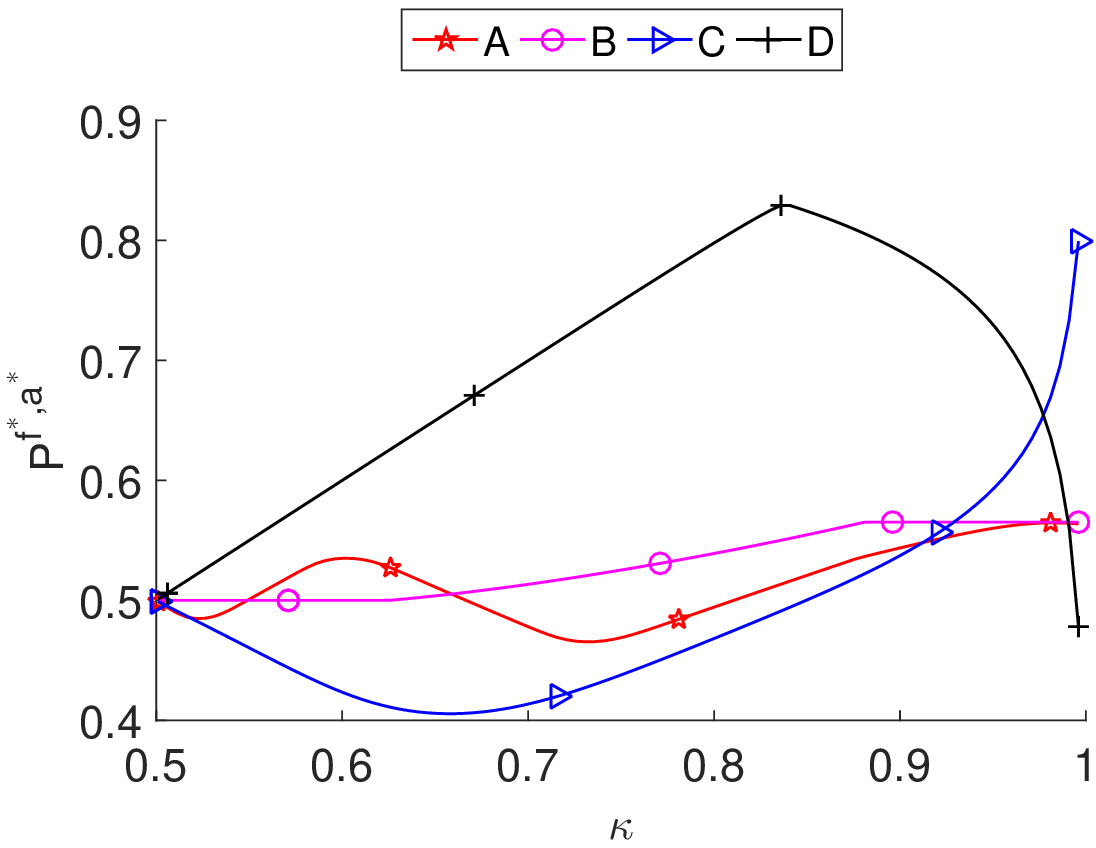}
\caption{}
\label{fig: Wild P1star graph}
\end{minipage}
\end{figure}

\section{Proof of Proposition \ref{atom equilib bets prop}}\label{prop 2 proof app}

We only prove (i). The argument for (ii) is similar, and as observed above, \ref{atom bet prop orig iii} is due to Isaacs (\cite{isaacs}).

Suppose that (\ref{atom bets out1 ineq alone}) holds. Then 
\begin{equation}\label{key isaacs ineq for prop}
\displaystyle\frac{1-q}{d_2} < \displaystyle\frac{1- \kappa q}{d_2} < \displaystyle\frac{1}{d_1 + d_2} < \displaystyle\frac{1}{\kappa \left( d_1 + d_2 \right) }  < \displaystyle\frac{q}{d_1}. 
\end{equation}
We get the first and third inequalities because $\kappa \in \left( 0 , 1\right)$. The second inequality is due to (\ref{1 bc impli}), while the last inequality is a rearrangement of (\ref{atom bets out1 ineq alone}). For now, the critical observation is that the leftmost quantity is less than the rightmost, allowing us to use the work of Isaacs (\cite{isaacs}).

Define the map
\begin{equation*}
\Phi : \mathbb{R}^2_{\geq 0} \longrightarrow \mathbb{R}
\end{equation*}
by 
\begin{align*}
\Phi \left( b_1 , b_2 \right) &= b_1 \left( \displaystyle\frac{\kappa \left( b_1 + d_1 + b_2 + d_2 \right) q}{b_1 + d_1}  - 1 \right)  + b_2 \left( \displaystyle\frac{\kappa \left( b_1 + d_1 + b_2 + d_2 \right) \left( 1  - q \right) }{b_2 + d_2}  - 1 \right)  . 
\end{align*}
$\Phi \left( b_1 , b_2 \right)$ is the atomic player's expected profit, given that she wagers $b_i$ on Outcome $i$, and it allows us to rewrite (\ref{atom max prob eqn prop}) as 
\begin{equation*}
 a_1   \bigg( \displaystyle\frac{ \kappa  \,  q }{ P^{f , a }} - 1 \bigg) + a_2  \left( \displaystyle\frac{ \kappa  \,  \left( 1 - q\right) }{  1 - P^{f , a } } - 1 \right)  = \displaystyle\sup_{\substack{   b_1 , b_2 \geq 0 \\ b_1 + b_2 \leq w }} \Phi \left( b_1 , b_2 \right)  . 
\end{equation*}
Technically, strategy profiles $\left( b_1 , b_2 \right)$ with
\begin{equation*}
 b_1 + b_2 > w
 \end{equation*}
 are not feasible, but this is unimportant.

Isaacs showed that $\Phi$ has a unique global maximum on $\mathbb{R}^2_{\geq 0}$, denoted $\left( b_1^{\star} , b_2^{\star}  \right)$, given by
\begin{equation*}
\left( b_1^{\star} , b_2^{\star}  \right)  = \left( \displaystyle\sqrt{ \displaystyle\frac{ \kappa  q d_1 d_2}{1 - \kappa q }  } -  d_1 , 0 \right) .
\end{equation*}
This proves that (i) holds when $w \geq b_1^{\star}$.

 
 The case where $w < b_1^{\star}$ is handled with elementary calculus. First, observe that (\ref{atom bets out1 ineq alone}) implies that $q > 0$ and 
\begin{align*}
\partial_{b_1 } \Phi \left( 0 , 0 \right) = \displaystyle\frac{ \kappa \left( d_1 + d_2 \right) q}{ d_1  }  - 1 > 0 .
\end{align*}
Since $\partial_{b_1 } \Phi \left(  b_1^{\star}  , 0 \right) = 0$ and 
\begin{equation*}
\partial_{b_1 b_1} \Phi \left( b_1 , b_2 \right) = - \displaystyle\frac{ 2 \kappa d_1 \left( b_2 + d_2 \right) q}{\left( b_1 + d_1 \right)^3}
\end{equation*}
is always negative, it follows that $\partial_{b_1} \Phi \left( b_1  , 0 \right) > 0$ for $0 \leq b_1 < b_1^{\star}$. The interpretation is that as long as the atomic player has not yet wagered on Outcome 2, she should wager as much as she can up to $b_1^{\star}$ on Outcome 1.

We only need to argue that she should never wager on Outcome 2, that is, $\partial_{b_2} \left( b_1 , b_2 \right)$ is negative whenever $b_1 + b_2 \leq w$. Rearranging (\ref{key isaacs ineq for prop}) implies that 
\begin{equation*}
 \partial_{b_2 } \Phi \left( 0, 0 \right) = \displaystyle\frac{ \kappa  \left( d_1 + d_2 \right) \left( 1 - q \right)}{ d_2  }  - 1 < 0 . 
\end{equation*}
Now $\partial_{b_2 } \Phi \left(  b_1^\star , 0 \right) \leq 0$\footnote{We know that $\partial_{b_1 } \Phi \left(  b_1^{\star}  , 0 \right) = 0$ and $\partial_{b_2 } \Phi \left(  b_1^\star , 0 \right) \leq 0$ since $\Phi$ has a unique global maximum on $\mathbb{R}^2_{\geq 0}$ at $\left( b_1^{\star} , b_2^{\star}  \right)$.} and 
\begin{equation*}
\partial_{b_2 b_1 } \Phi \left( b_1 , b_2 \right) = \displaystyle\frac{ \kappa d_1  q }{ \left( b_1 + d_1 \right)^2} + \displaystyle\frac{ \kappa d_2 \left( 1 - q \right) }{ \left( b_2 + d_2 \right)^2} 
\end{equation*}
is positive everywhere, so $\partial_{b_2} \Phi \left( b_1 , 0 \right)< 0$ for $0 \leq b_1 < b_1^\star$. Since 
\begin{equation*}
\partial_{b_2 b_2} \Phi \left( b_1 , b_2 \right) = - \displaystyle\frac{ 2 \kappa d_2 \left( b_1 + d_1 \right) \left( 1 - q \right)}{\left( b_2 + d_2 \right)^3}
\end{equation*}
is always non-positive, we conclude that $\partial_{b_2} \left( b_1 , b_2 \right)$ must be negative whenever $0 \leq b_1 < b_1^\star$. In particular, $\partial_{b_2} \left( b_1 , b_2 \right)$ is negative, if $ b_1 + b_2 \leq w$.

\qed

\section{Proof of Theorem \ref{main thm}}\label{main thm proof app}

{\bf Step \mystep{step 1 orig}: If $\left(f^{\star}, a^{\star} \right)$ is an equilibrium, then $d_1^{\star}$, $d_2^{\star} > 0$. } 

\vspace{2mm}

We show that the total amount wagered on each outcome by the diffuse players is positive in equilibrium: $d_1^{\star}$, $d_2^{\star} > 0$. We use this result to complete the proof of the {\it only if} direction of Theorem \ref{main thm} in Step \ref{step 2 orig}. It also allows us to use Propositions \ref{diff equilib bets prop} and \ref{atom equilib bets prop} in the proof of the {\it if} direction.

Suppose instead that we can find a pure-strategy Nash equilibrium $\left( f^{\star} , a^{\star} \right)$ with $d_1^{\star} = 0$. It follows immediately that $d_2^{\star} > 0$: otherwise, \ref{NE def orig iii} of Definition \ref{maj min NE def} implies that $a_1^{\star}= a_2^{\star}= 0$, which contradicts \ref{NE def orig i} of Definition \ref{maj min NE def}. A consequence of Isaacs' work is that $q = 0$ (\cite{isaacs}).\footnote{If $q > 0$, then  $a_1^{\star}$ is undefined. Roughly, the atomic player needs to make an arbitrarily small bet on Outcome 1 but is not allowed to do so.} Then \ref{NE def orig iii} of Definition \ref{maj min NE def} implies that $a_1^{\star}= a_2^{\star}= 0$. In particular, $P^{f^{\star}, a^{\star}} = 0$. By \ref{NE def orig i} of Definition \ref{maj min NE def}, $f_2^{\star} \equiv 0$, which is impossible since $d_2^{\star} > 0$. 

Hence, $d_1^{\star} > 0$. It follows similarly that $d_2^{\star} > 0$.

\vspace{2mm}

\noindent \textbf{Step \mystep{step 2 orig}: If an equilibrium exists, then $\kappa >  0.5$.} 

\vspace{2mm}

We finish the proof of the {\it only if} direction of Theorem \ref{main thm} by formalizing the heuristics given previously. Suppose that $\left( f^{\star} , a^{\star} \right)$ is an equilibrium. Step \ref{step 1 orig} implies that $d_1^{\star}$, $d_2^{\star} > 0$. Then both 
\begin{equation}\label{Pstar in 1 min al al}
\displaystyle\frac{P^{f^{\star}, a^{\star} }}{ \kappa } < 1 \qquad \text{and} \qquad \displaystyle\frac{1 - P^{f^{\star}, a^{\star} } }{ \kappa }  < 1 
\end{equation}
by Proposition \ref{diff equilib bets prop}. Rearranging (\ref{Pstar in 1 min al al}) finishes the argument.

\vspace{2mm}

\noindent \textbf{Step \mystep{step 3 alt 2}: Definition and discussion of $\bar{p}_i$ (when $\kappa >  0.5$).}

\vspace{2mm}

Due to Step \ref{step 2 orig}, we assume that $\kappa >  0.5$ for the remainder of the proof (Steps \ref{step 3 alt 2} - \ref{step 5 orig}). This assumption ensures that our discussions are not vacuous, as we implicitly rely on the positive length of the interval $\left[ 1- \kappa, \kappa \right]$.

We now define and discuss the quantities $\bar{p}_1$, $\bar{p}_2 \in \left[ 1- \kappa , \kappa \right]$. We use this notation when we define our $\zeta_i$ maps in Step \ref{step 4 alt 2} and $\varphi$ in Step \ref{step 5 alt 2}. Roughly, $\bar{p}_1$ is a na\"{\i}ve lower bound for the implied probability of Outcome 1 when the atomic player wagers on Outcome 1. Similarly, $\bar{p}_2$ is an upper bound for the implied probability of Outcome 1 when the atomic player wagers on Outcome 2. Both are derived from Propositions \ref{diff equilib bets prop} and \ref{atom equilib bets prop}.

Since $\mu$'s density is positive, the map 
\begin{equation}\label{out1 bc map}
p \mapsto \displaystyle\frac{  \mu \left( \frac{ p  }{\kappa} , 1 \right]   }{  \kappa \left( \mu \left( \frac{ p  }{\kappa} , 1 \right]      + \mu  \left[0, 1 - \frac{1 -p }{\kappa}  \right)   \right)  }
\end{equation}
is decreasing and continuous on $\left[ 1 - \kappa , \kappa \right]$. Its value is $1/ \kappa$ at $p = \left( 1- \kappa \right)$ and 0 at $p = \kappa$. Hence, there is a unique $\bar{p}_1 \in \left( 1 - \kappa , \kappa \right]$ such that 
\begin{equation}\label{pbar1 defn eqn}
q = \displaystyle\frac{  \mu \left( \frac{ \bar{p}_1  }{\kappa} , 1 \right]   }{  \kappa \left( \mu \left( \frac{ \bar{p}_1   }{\kappa} , 1 \right]      + \mu  \left[0, 1 - \frac{1 - \bar{p}_1  }{\kappa}  \right)   \right)  } . 
\end{equation}

Clearly, $\bar{p}_1 = \kappa$, or equivalently,  $\left( \bar{p}_1 , \kappa \right]$ is empty, if and only if $q = 0$. Regardless of $q$'s value, we know 
\begin{align}\label{q ineq pbar1 ints}
\left\{ \begin{array}{cc}
q \leq \displaystyle\frac{  \mu \left( \frac{ p  }{\kappa} , 1 \right]   }{  \kappa \left( \mu \left( \frac{ p  }{\kappa} , 1 \right]      + \mu  \left[0, 1 - \frac{1 -p }{\kappa}  \right)   \right)  } \qquad & \qquad \text{if} \quad p \in \left[ 1- \kappa, \bar{p}_1 \right]  \\
& \quad \\
q > \displaystyle\frac{  \mu \left( \frac{ p  }{\kappa} , 1 \right]   }{  \kappa \left( \mu \left( \frac{ p  }{\kappa} , 1 \right]      + \mu  \left[0, 1 - \frac{1 -p }{\kappa}  \right)   \right)  } \qquad & \qquad \text{if} \quad p \in \left( \bar{p}_1 , \kappa \right]  
\end{array} \right. .
\end{align}
We later connect this observation to (\ref{atom bets out1 ineq alone}) and (\ref{atom bets noth ineq alone}) when defining $\zeta_i$ and $\varphi$.

Similarly, the map 
\begin{equation}\label{out2 bc map}
p \mapsto \displaystyle\frac{  \mu  \left[0, 1 - \frac{1 -p }{\kappa}  \right)    }{  \kappa \left( \mu \left( \frac{ p  }{\kappa} , 1 \right]      + \mu  \left[0, 1 - \frac{1 -p }{\kappa}  \right)   \right)  }
\end{equation}
is increasing and continuous on $\left[ 1 - \kappa , \kappa \right]$. Its value is 0 at $p = \left( 1- \kappa \right)$ and $1/ \kappa$ at $p = \kappa$, so there is a unique $\bar{p}_2 \in \left[ 1 - \kappa , \kappa \right)$ such that 
\begin{equation*}
1 - q = \displaystyle\frac{  \mu  \left[0, 1 - \frac{1 - \bar{p}_2  }{\kappa}  \right) }{  \kappa \left( \mu \left( \frac{ \bar{p}_2   }{\kappa} , 1 \right]      + \mu  \left[0, 1 - \frac{1 - \bar{p}_2  }{\kappa}  \right)   \right)  } . 
\end{equation*}

Now $\bar{p}_2 =  \left( 1-\kappa \right)$, i.e.,  $\left[ 1- \kappa , \bar{p}_2 \right)$ is empty, if and only if $q = 1$. In any case,
\begin{align}\label{q ineq pbar2 ints}
\left\{ \begin{array}{cc}
1 - q > \displaystyle\frac{  \mu  \left[0, 1 - \frac{1 -p }{\kappa}  \right)    }{  \kappa \left( \mu \left( \frac{ p  }{\kappa} , 1 \right]      + \mu  \left[0, 1 - \frac{1 -p }{\kappa}  \right)   \right)  } \qquad & \qquad \text{if} \quad p \in \left[ 1 - \kappa,  \bar{p}_2 \right)  \\
&\\
1 - q \leq \displaystyle\frac{  \mu  \left[0, 1 - \frac{1 -p }{\kappa}  \right)    }{  \kappa \left( \mu \left( \frac{ p  }{\kappa} , 1 \right]      + \mu  \left[0, 1 - \frac{1 -p }{\kappa}  \right)   \right)  } \qquad & \qquad \text{if} \quad p \in \left[  \bar{p}_2  , \kappa \right]  
\end{array} \right.  .
\end{align}
This comment relates to (\ref{atom bets out2 ineq alone}) and (\ref{atom bets noth ineq alone}), as reflected in our definitions of $\zeta_i$ and $\varphi$.

We conclude by observing that we have $\bar{p}_2 < \bar{p}_1$ since $\kappa \in \left(  0.5 , 1  \right)$. This is another key remark for our future definition of $\varphi$.

\vspace{2mm}

\noindent \textbf{Step \mystep{step 4 alt 2}: Definition and discussion of $\zeta_i$ (when $\kappa >  0.5$).}

\vspace{2mm}

We define and discuss the functions $\zeta_1$ and $\zeta_2$. We use this notation in our definition of $\varphi$ in Step \ref{step 5 alt 2}. Intuitively, $\zeta_i \left( p \right)$ represents the amount that the atomic player wagers on Outcome $i$ when she makes a positive wager on Outcome $i$, does not face a budget constraint, and the implied probability of Outcome 1 is $p$ (cf. (\ref{atom bets out1 soln alone}) and (\ref{atom bets out2 soln alone})).

For $p \in \left[ \bar{p}_1 , \kappa \right]$, define $\zeta_1$ by 
\begin{equation*}
\zeta_1 \left(  p \right) =  \sqrt{\frac{\kappa q}{1 - \kappa q} \,  \mu  \left( \frac{p}{\kappa} , 1 \right]  \mu  \left[0, 1 - \frac{1 -p}{\kappa}  \right)   }  - \mu  \left( \frac{p}{\kappa} , 1 \right] . 
\end{equation*}
Since $\mu$ has a positive density, (\ref{q ineq pbar1 ints}) implies that 
\begin{equation*}
\mu  \left( \frac{p}{\kappa} , 1 \right]  > 0 , \qquad \mu  \left[0, 1 - \frac{1 -p}{\kappa}  \right)  > 0 , \qquad \text{and} \qquad q > \displaystyle\frac{  \mu \left( \frac{ p  }{\kappa} , 1 \right]   }{  \kappa \left( \mu \left( \frac{ p  }{\kappa} , 1 \right]      + \mu  \left[0, 1 - \frac{1 -p }{\kappa}  \right)   \right)  }
\end{equation*}
for $p \in \left( \bar{p}_1 , \kappa \right)$. It follows as in (\ref{1 bc impli}) that $\zeta_1$ is positive on $\left( \bar{p}_1 , \kappa \right)$. We also have 
\begin{equation*}
\zeta_1 \left(  \bar{p}_1 \right) = \zeta_1 \left(  \kappa \right) = 0
\end{equation*}
 from (\ref{pbar1 defn eqn}).

We define $\zeta_2$ for $p \in \left[ 1 - \kappa,  \bar{p}_2 \right]$ by 
\begin{equation*}
\zeta_2 \left(  p \right) =  \sqrt{\frac{\kappa \left( 1 - q\right) }{1 - \kappa \left( 1 - q \right)} \,  \mu  \left( \frac{p}{\kappa} , 1 \right]  \mu  \left[0, 1 - \frac{1 -p}{\kappa}  \right)   }  -  \mu  \left[0, 1 - \frac{1 -p}{\kappa}  \right)  . 
\end{equation*}
Using the techniques from our discussion of $\zeta_1$, we see that $\zeta_2$ is positive on $\left( 1 - \kappa,  \bar{p}_2 \right)$ and 
\begin{equation*}
 \zeta_2 \left( 1-  \kappa \right) = \zeta_2 \left(  \bar{p}_2 \right) = 0 . 
\end{equation*}

\vspace{2mm}

\noindent \textbf{Step \mystep{step 5 alt 2}: Definition and discussion of $\varphi$ (when $\kappa >  0.5$).}

\vspace{2mm}

We introduce the {\it implied probability map} $\varphi$. We see in Steps \ref{step 4 orig} and \ref{step 5 orig} that a fixed-point of $\varphi$ corresponds to a pure-strategy Nash equilibrium in an obvious way and vice versa, which ultimately allows us to complete Theorem \ref{main thm}'s proof.

$\varphi$'s domain is the set of candidates $p$ for the implied probability of Outcome 1. We need only consider $p \in \left[ 1- \kappa, \kappa \right]$, as we observe in Step \ref{step 5 orig}. Proposition \ref{diff equilib bets prop} says that
\begin{align}\label{diff bet examps step 3}
\mu  \left( \frac{p}{\kappa} , 1 \right]  \qquad \text{and} \qquad \mu  \left[0, 1 - \frac{1 -p}{\kappa}  \right) 
\end{align}
are the total amounts wagered by the diffuse players on Outcomes 1 and 2, respectively. The atomic player's wagers are described by Proposition \ref{atom equilib bets prop}. For example, if 
\begin{equation*}
 q > \displaystyle\frac{  \mu \left( \frac{ p  }{\kappa} , 1 \right]   }{  \kappa \left( \mu \left( \frac{ p  }{\kappa} , 1 \right]      + \mu  \left[0, 1 - \frac{1 -p }{\kappa}  \right)   \right)  } , 
\end{equation*}
that is, $p \in \left( \bar{p}_1 , \kappa \right]$, then the atomic player wagers nothing on Outcome 2 and 
\begin{equation*}
\displaystyle\min \left\{ w  , \zeta_1 \left( p \right) \right\}
\end{equation*}
on Outcome 1. Recalculating the implied probability of Outcome 1 using these bets, we get 
\begin{equation*}
\displaystyle\frac{ \displaystyle\min \left\{  w , \zeta_1 \left( p \right) \right\} + \mu  \left( \frac{p}{\kappa} , 1 \right]  }{ \displaystyle\min \left\{ w , \zeta_1 \left( p \right) \right\} + \mu  \left( \frac{p}{\kappa} , 1 \right]    + \mu  \left[0, 1 - \frac{1 -p}{\kappa}  \right)  } 
\end{equation*} 
from Definition \ref{impl prob defn}. We set $\varphi \left( p \right)$ to this value. In some sense, this procedure is only potentially meaningful when $p$ is equal to $\varphi \left( p \right)$, leading to our focus on fixed-points. 

Here is the complete definition of $\varphi$ suggested by this explanation:
\begin{align*}
\varphi \left( p \right) = \left\{ \begin{array}{cc}
\displaystyle\frac{\mu  \left( \frac{p}{\kappa} , 1 \right]  }{ \displaystyle\min \left\{ w , \zeta_2 \left( p \right) \right\} + \mu  \left( \frac{p}{\kappa} , 1 \right]    + \mu  \left[0, 1 - \frac{1 -p}{\kappa}  \right)  } & \quad \text{if} \quad p \in \left[ 1 - \kappa,  \bar{p}_2 \right) \\
&\\
\displaystyle\frac{  \mu  \left( \frac{p}{\kappa} , 1 \right]  }{  \mu  \left( \frac{p}{\kappa} , 1 \right]    + \mu  \left[0, 1 - \frac{1 -p}{\kappa}  \right)  } & \quad \text{if} \quad p \in \left[  \bar{p}_2 , \bar{p}_1 \right] \\
&\\
\displaystyle\frac{ \displaystyle\min \left\{ w , \zeta_1 \left( p \right) \right\} + \mu  \left( \frac{p}{\kappa} , 1 \right]  }{ \displaystyle\min \left\{ w , \zeta_1 \left( p \right) \right\} + \mu  \left( \frac{p}{\kappa} , 1 \right]    + \mu  \left[0, 1 - \frac{1 -p}{\kappa}  \right)  } &  \quad \text{if} \quad p \in \left( \bar{p}_1 , \kappa \right]
\end{array} \right. .
\end{align*}

Observe that $\varphi$ is continuous on $\left[ 1- \kappa , \kappa \right]$ since $\mu$ has a positive density and $\zeta_i \left( \bar{p}_i \right) = 0$ (see Step \ref{step 4 alt 2}). This helps us prove that $\varphi$ has a unique fixed-point in Step \ref{step 6 uniq fp}.

\vspace{2mm}

\noindent \textbf{Step \mystep{step 6 alt varphi dec}: $\varphi$ is decreasing (when $\kappa >  0.5$).} 

\vspace{2mm}

We show that $\varphi$ is decreasing. We use this property in Step \ref{step 6 uniq fp} to argue that $\varphi$ has a unique fixed-point.

Since $\mu$'s density is positive everywhere, $\varphi$ is decreasing on $\left[ \bar{p}_2 , \bar{p}_1 \right]$. It suffices to show that $\varphi$ is decreasing on both $\left[ 1 - \kappa,  \bar{p}_2 \right)$ and $\left( \bar{p}_1 , \kappa \right]$ because $\varphi$ is continuous. The proofs are similar, so we consider the former case. 

We are done, if $\left[ 1 - \kappa,  \bar{p}_2 \right)$ is empty. Suppose that it is not. Recall from Step \ref{step 3 alt 2} that this is equivalent to assuming that $q < 1$. Define two functions $\varphi_2$ and $\varphi_2^w$ on $\left[ 1 - \kappa,  \bar{p}_2 \right)$ by 
\begin{align*}
\varphi_2 \left( p \right) &= \displaystyle\frac{\mu  \left( \frac{p}{\kappa} , 1 \right]  }{ \zeta_2 \left( p \right)  + \mu  \left( \frac{p}{\kappa} , 1 \right]    + \mu  \left[0, 1 - \frac{1 -p}{\kappa}  \right)  } \\
\varphi_2^w \left( p \right) &= \displaystyle\frac{\mu  \left( \frac{p}{\kappa} , 1 \right]  }{ w+ \mu  \left( \frac{p}{\kappa} , 1 \right]    + \mu  \left[0, 1 - \frac{1 -p}{\kappa}  \right)  } .
\end{align*}
The point is that 
\begin{equation*}
\varphi \left( p \right) = \displaystyle\max \left\{ \varphi_2 \left( p \right) , \varphi_2^w \left( p \right) \right\} 
\end{equation*}
on $\left[ 1 - \kappa,  \bar{p}_2 \right)$, so it is enough to show that $\varphi_2$ and $\varphi_2^w$ are both decreasing.

Clearly, $\varphi_2^w$ is decreasing. Denoting the positive and continuous density of $\mu$ by $g$, we see that $\varphi_2$ is decreasing because 
\begin{align*}
\varphi_{2}^{\prime} \left( p \right) &= -  \displaystyle\frac{  g \left( \frac{p}{\kappa} \right) \sqrt{\frac{\kappa \left( 1 -q \right)}{1 - \kappa \left( 1 -q \right)}   \mu  \left( \frac{p }{\kappa} , 1 \right] \mu  \left[0, 1 - \frac{1 -p}{\kappa}  \right)  } }{ 2\kappa  \left( \sqrt{\frac{\kappa \left( 1 -q \right)}{1 - \kappa \left( 1 -q \right)}   \mu  \left( \frac{p }{\kappa} , 1 \right]    \mu   \left[0, 1 - \frac{1 -p}{\kappa}  \right)    }+ \mu   \left( \frac{p }{\kappa} , 1 \right]     \right)^2 } \\
&\qquad - \displaystyle\frac{ \left( \frac{\kappa \left( 1 -q \right)}{1 - \kappa \left( 1 -q \right)} \right)^{ 0.5} \mu  \left( \frac{p}{\kappa} , 1 \right]^{3/2} \mu  \left[0, 1 - \frac{1 -p}{\kappa}  \right)^{- 0.5} g \left( 1-  \frac{1- p}{\kappa} \right) }{ 2\kappa  \left( \sqrt{\frac{\kappa \left( 1 -q \right)}{1 - \kappa \left( 1 -q \right)}   \mu  \left( \frac{p }{\kappa} , 1 \right]  \mu   \left[0, 1 - \frac{1 -p}{\kappa}  \right)  }+ \mu  \left( \frac{p }{\kappa} , 1 \right]   \right)^2 } . 
\end{align*}

\vspace{2mm}

\noindent \textbf{Step \mystep{step 6 uniq fp}: $\varphi$ has a unique fixed-point (when $\kappa >  0.5$).} 

\vspace{2mm}

 We show that $\varphi$ has a unique fixed-point. We use the existence of $\varphi$'s fixed-point to prove the existence of a pure-strategy Nash equilibrium in Step \ref{step 4 orig}, while we use the uniqueness of $\varphi$'s fixed-point to demonstrate that the equilibrium is unique in Step \ref{step 5 orig}.

The proof is simple: $\varphi \left( 1 - \kappa \right) = 1$ and $\varphi \left( \kappa \right) = 0$. Since $\varphi$ is continuous and decreasing (see Steps \ref{step 5 alt 2} - \ref{step 6 alt varphi dec}), it must have a unique fixed-point.

\vspace{2mm}

\noindent \textbf{Step \mystep{step 4 orig}: An equilibrium exists (when $\kappa >  0.5$).} 

\vspace{2mm}

We show that a pure-strategy Nash equilibrium exists. Based on Step \ref{step 6 uniq fp}, we need only describe how to construct an equilibrium from a fixed-point of $\varphi$.

Suppose that $\hat{P}$ is a fixed-point of $\varphi$. The proofs for each case are similar, so we only present the argument when $\hat{P} \in \left[ 1 - \kappa , \bar{p}_2 \right)$. Define a feasible strategy profile $\left( f , a \right)$ by 
\begin{align*}
f_1 \left( p \right) &= \left\{ \begin{array}{cc}
1 & \quad \text{if } \, \,  p \geq \hat{P}  / \kappa \\
&\\
0 & \quad \text{if } \, \, p < \hat{P}  / \kappa \\
\end{array}   \right. , \qquad \qquad
f_2 \left( p \right) = \left\{ \begin{array}{cc}
1 & \quad \text{if } \, \, 1 - p  \geq \left( 1 - \hat{P} \right) / \kappa  \\
&\\
0 & \quad \text{if } \, \, 1 - p  < \left( 1 - \hat{P} \right) / \kappa  \\
\end{array} \right. ,
\end{align*}
$a_1 = 0$, and
\begin{equation*}
a_2 = \displaystyle\min \left\{  w , \zeta_2 \left( \hat{P} \right) \right\} . 
\end{equation*}
In particular, 
\begin{equation*}
d_1 = \mu  \left( \frac{\hat{P} }{\kappa} , 1 \right]  \qquad \text{and} \qquad d_2 = \mu  \left[0, 1 - \frac{1 - \hat{P}}{\kappa}  \right) .
\end{equation*}

(\ref{q ineq pbar2 ints}) implies that (\ref{atom bets out2 ineq alone}) is satisfied. Since $\varphi \left( 1 - \kappa  \right) = 1$, we know that $\hat{P} \ne \left( 1 - \kappa  \right)$. Hence, $d_1$, $d_2 > 0$ because the density of $\mu$ is positive everywhere. We then have \ref{NE def orig iii} of Definition \ref{maj min NE def} by Proposition \ref{atom equilib bets prop}.

Since $\hat{P}$ is a fixed-point of $\varphi$,
\begin{equation*}
\hat{P} = \frac{ \mu  \left( \frac{\hat{P}}{\kappa} , 1 \right]  }{ \displaystyle\min \left\{  w , \zeta_2 \left( \hat{P} \right) \right\}  + \mu  \left( \frac{\hat{P}}{\kappa} , 1 \right]    + \mu  \left[0, 1 - \frac{1 -\hat{P}}{\kappa}  \right)  } = P^{f,a}. 
\end{equation*}
From Proposition \ref{diff equilib bets prop}, we see that \ref{NE def orig ii} of Definition \ref{maj min NE def} holds. 

This completes the proof, as (i) of Definition \ref{maj min NE def} is obviously satisfied.

\vspace{2mm}

\noindent \textbf{Step \mystep{step 5 orig}: The equilibrium in Step \ref{step 4 orig} is unique (when $\kappa >  0.5$).} 

\vspace{2mm}

We conclude Theorem \ref{main thm}'s proof by showing that the equilibrium in Step \ref{step 4 orig} is unique. The key observation is that $\varphi$'s fixed-point is also unique (see Step \ref{step 6 uniq fp}).

First, we describe how to construct a fixed-point of $\varphi$, given an equilibrium $\left( f^{\star} , a^{\star} \right)$. From Proposition \ref{diff equilib bets prop} and Step \ref{step 1 orig}, we know that
\begin{equation*}
d_1^{\star} = \mu  \left( \frac{P^{f^{\star}, a^{\star}}  }{\kappa} , 1 \right] > 0 \qquad \text{and} \qquad d_2^{\star} = \mu  \left[0, 1 - \frac{1 - P^{f^{\star}, a^{\star}}  }{\kappa}  \right) > 0 . 
\end{equation*}
In particular, $P^{f^{\star}, a^{\star}}  \in \left( 1 - \kappa , \kappa \right)$. By Proposition \ref{atom equilib bets prop}, there are three possibilities for $a^{\star}$.

Assume that $a_1^{\star} > 0$ and $a_2^{\star} = 0$. The other cases can be handled similarly. By Proposition \ref{atom equilib bets prop}, (\ref{atom bets out1 ineq alone}) holds. Hence, $P^{f^{\star}, a^{\star}} \in \left( \bar{p}_1 , \kappa \right)$ due to (\ref{q ineq pbar1 ints}) and 
\begin{align*}
a_1^{\star} = \displaystyle\min \left\{ w ,  \sqrt{ \displaystyle\frac{ \kappa  q d_1^{\star} d_2^{\star} }{1 - \kappa q }  } -  d_1^{\star} \right\} =  \displaystyle\min \left\{ w , \zeta_1 \left( P^{f^{\star}, a^{\star}}  \right) \right\} . 
\end{align*}
By Definition \ref{impl prob defn}, 
 \begin{align*}
 P^{f^{\star}, a^{\star}} &= \displaystyle\frac{ \displaystyle\min \left\{ w , \zeta_1 \left(  P^{f^{\star}, a^{\star}} \right) \right\} + \mu  \left( \frac{ P^{f^{\star}, a^{\star}}}{\kappa} , 1 \right]  }{ \displaystyle\min \left\{ w , \zeta_1 \left(  P^{f^{\star}, a^{\star}} \right) \right\} + \mu  \left( \frac{ P^{f^{\star}, a^{\star}}}{\kappa} , 1 \right]    + \mu  \left[0, 1 - \frac{1 - P^{f^{\star}, a^{\star}}}{\kappa}  \right)  } \\
 &= \varphi \left(  P^{f^{\star}, a^{\star}} \right), 
  \end{align*} 
 that is, $ P^{f^{\star}, a^{\star}}$ is a fixed-point of $\varphi$.

 Now suppose that we have another equilibrium $\left( f^{\diamond} , a^{\diamond} \right)$. Using the method just described, we see that $P^{f^{\diamond}, a^{\diamond}}$ is a fixed-point of $\varphi$. Since there is exactly one fixed-point of $\varphi$ by Step \ref{step 6 uniq fp}, $P^{f^{\diamond}, a^{\diamond}} = P^{f^{\star}, a^{\star}}$. 
 
 By Step \ref{step 1 orig} and Proposition \ref{diff equilib bets prop}, $f^{\star}$ and $f^{\diamond}$ necessarily agree everywhere, except when
\begin{align*}
p = P^{f^{\star}, a^{\star}} / \kappa \qquad \text{or} \qquad   1- p = \left( 1 - P^{f^{\star}, a^{\star}} \right) / \kappa .
\end{align*}
Since $\mu$ has a density, it follows that $f^{\star} = f^{\diamond}$ $\mu$-a.s. Clearly, $d_1^{\star} = d_1^{\diamond}$ and $d_2^{\star} = d_2^{\diamond}$. Proposition \ref{atom equilib bets prop} implies that $a^{\star} = a^{\diamond}$.

\qed

\bibliographystyle{siam}
\bibliography{MFGBib}

\end{document}